\documentclass[11pt]{article}

\usepackage[letterpaper, top=1in, bottom=1in, left=0.75in, right=0.75in]{geometry}
\usepackage{amsfonts, amsmath, amstext, amssymb, amsthm}
\usepackage{graphicx,subfig,placeins}
\usepackage{enumerate}
\usepackage{color}
\usepackage{cite}
\usepackage{acronym}
\usepackage{accents}

\newtheorem{lemma}{Lemma}
\newtheorem{theorem}{Theorem}
\newtheorem{corollary}{Corollary}
\newtheorem{proposition}{Proposition}
\theoremstyle{definition}
\newtheorem{assumption}{Assumption}

\newcommand{\reals}{\mathbb{R}}
\newcommand{\naturals}{\mathbb{N}}
\newcommand{\K}{\mathcal{K}}
\newcommand{\KL}{\mathcal{KL}}
\newcommand{\cost}{\mathcal{J}}
\newcommand{\trans}{{\top}}

\newcommand{\ropt}{\bar{r}^\star}

\newcommand{\trace}{\mathrm{Tr}}
\newcommand{\ubar}[1]{\underaccent{\bar}{#1}}

\acrodef{esc}[{\sc esc}]{extremum-seeking control}
\acrodef{iss}[{\sc iss}]{input-to-state stable}
\acrodef{pe}[{\sc pe}]{persistently exciting}
\acrodef{as}[{\sc as}]{asymptotically stable}
\acrodef{rls}[{\sc rls}]{recursive least-squares}
\acrodef{bls}[{\sc bls}]{batch least-squares}
\acrodef{pi-esc}[{\sc pi-esc}]{proportional-integral extremum-seeking controller}

\title{\small \textcolor{blue}{This work has been submitted to the IEEE for possible publication. Copyright may be transferred without notice, after which this version may no longer be accessible.} \\ \bigskip \huge Extremum Seeking Control with an Adaptive Gain Based On Gradient Estimation Error}
\author{Claus Danielson\thanks{Department of Mechanical Engineering, University of New Mexico.},  
Scott A. Bortoff, and Ankush Chakrabarty\thanks{Mitsubishi Electric Research Laboratories (MERL).  Corresponding author. Email: achakrabarty@ieee.org. Phone:
+1 (617) 758-6175.}
}
\date{}

\begin{document}

\maketitle

\begin{abstract}
This paper presents an \ac{esc} algorithm with an adaptive step-size that adjusts the aggressiveness of the controller based on the quality of the gradient estimate. The adaptive step-size ensures that the integral-action produced by the gradient descent does not destabilize the closed-loop system. To quantify the quality of the gradient estimate, we present a \ac{bls} estimator with a novel weighting and show that it produces bounded estimation errors, where the uncertainty is due to the curvature of the unknown cost function. The adaptive step-size then maximizes the decrease of the combined plant and controller Lyapunov function for the worst-case estimation error. We prove that our \ac{esc} controller is input-to-state stable with respect to the dither signal. Finally, we demonstrate our \ac{esc} controller through five numerical examples; one illustrative, one practical, and three benchmarks. 
\end{abstract}

\section{Introduction}
\acresetall

Extremum seeking control is a century-old~\cite{LeBlanc1922} form of model-free adaptive control for the real-time optimization of  dynamic systems. The objective of \ac{esc} is to drive the plant to an equilibrium that optimizes an unknown cost function. In this paper, we extend the \ac{pi-esc} from~\cite{ATTA2019147, Guay2015, guay2015time} using an adaptive step-size that adjusts the aggressiveness of the controller based on the quality of the gradient estimate. 

Most \ac{esc} controllers can be interpreted as gradient descent algorithms, wherein the controller follows a descent direction to the optimal. From the perspective of dynamic systems, gradient descent is integral-action. For instance, the standard gradient descent~\cite{ATTA2019147, Guay2015, guay2015time, Danielson2006, Biyik2007, Chakrabarty2020} wherein the updated set-point is the previous set-point minus a step-size times the gradient, has the dynamics of a discrete-time integrator. In classical \ac{esc}, the set-point is the continuous-time integral of the estimated gradients~\cite{Ariyur2003, Benosman2017}. Whether in discrete-time or continuous-time, it is a fundamental result from control-theory that integral-action can lead to instability. This issue is further complicated for non-linear systems. \ac{esc} controllers employ a variety of strategies to preserve stability. For instance, the inspiration for this paper~\cite{ATTA2019147, Guay2015, guay2015time} used a \ac{pi-esc} controller to preserved stability. In this paper, we introduce a integral gain matrix that ensured stability under the idealistic condition where the gradient is perfectly estimated. This integral gain represents the most aggressive step-size for the gradient descent that will not destabilize the plant. Our adaptive step-size attenuates this idealistic integral gain to preserve stability when the gradient estimate is imperfect. 

Typically, strategies for promoting stability requires slowing the integral-action of the gradient descent, which can result in a slower convergence to the optimal. To accelerate convergence, other mechanisms are often added to the \ac{esc} algorithm. For example, dither adaptation has been explored in~\cite{ATTA2019147,yin2019optimizing}, where the dither signal is made sufficiently small near the optimal solution so as to constrict the size of the uncertainty ball around the equilibrium state. Using the magnitude of the gradient for dither amplitude adaptation has been proposed in~\cite{wang2016stability}, and extended to dither adaptation using a super-twisting algorithm and higher-order sliding modes in~\cite{he2019twisting}. The problem of removing excitation signals in a phasor-based ESC that incorporates {\sc pi} control in the feedback path is discussed in~\cite{ATTA2019147}. The authors in~\cite{haring2016}  proposed a mechanism for reducing the dither amplitude in a traditional perturbation-based ESC strategy and provided a stability analysis. Methods for fast ESC convergence with high-frequency dither signals for systems with unknown dynamics is proposed in~\cite{Scheinker2013}, and with unknown Hessians in~\cite{Ghaffari2014}. Dither-free methods have also been explored in~\cite{Hunnekens2014}, and event-triggered mechanisms for fast convergence in~\cite{Poveda2017}. In contrast, this paper focuses on accelerating convergence by using an adaptive step-size in the integral-action of the gradient descent.  While dither perturbation has gained widespread attention, directly adapting the \ac{esc} controller gain is relatively uncommon, with a few exceptions, namely~\cite{salsbury2020self,Grushkovskaya2018,Suttner2019}. 

One of the main challenges of \ac{esc} is that the gradient of the unknown cost function must be estimated from data gathered while the system is in operation. Misestimating the gradient can exacerbate the stability issues introduced by the integral-action, especially when the gradient is over-estimated. This issue can be addressed by analyzing the \ac{esc} controller and gradient estimator in a common framework. For instance, \ac{esc} controllers often use an \ac{rls} estimator to estimate the gradient of the cost function~\cite{ATTA2019147, Guay2015, guay2015time, Danielson2006, Chakrabarty2020}. This approach is attractive since the gradient estimator has state dynamics that can be analyzed in a common Lyapunov framework with the plant and \ac{esc} controller dynamics. The disadvantage of this approach is that it often results in conservatively slowing integral-action to allow the estimator to converge. In contrast, our \ac{esc} controller employs a \ac{bls} estimator with a novel weighting. The \ac{bls} estimator can be viewed as an operator that maps batches of collected data to gradient estimates. Thus, we do not need to consider its convergence rate in our analysis. Instead, we show that the novel weightings used in our \ac{bls} estimator produce bounded estimation errors. The adaptive step-size uses these bounds to maximize the decrease of the joint plant and controller Lyapunov function for the worst-case gradient estimation error. The potential advantage of this approach is that we can use more a aggressive integral-action on average without risking instability. This advantage is empirically demonstrated through our benchmark simulations. Another advantage of this approach is that it will facilitate future work based on more general gradient estimators, such as moving horizon estimators or set-based estimators.

For our \ac{esc} algorithm, persistently exciting data is necessary, but not sufficient, to accurately estimate the gradient of the cost. For accurate estimates, the data gathered must also be sufficiently close to equilibrium. To quantify the distance from equilibrium, we assume that our plant is instrumented to provide addition measurements beyond the cost, which is the only measurement used in most \ac{esc} problem formulations. Furthermore, we assume that the cost is a static function of these measurements. This is an admittedly strong assumption for \ac{esc}, but one that is consistent with many industrial systems which are heavily instrumented. Exploiting these additional sensor measurements to improve the convergence is a shrewd strategy. We show that accurate estimates of the gradient requires persistently exciting and sufficiently small perturbations of the input. Thus, the dither based acceleration methods described above can potentially be combined with our adaptive gain to further improve convergence. 

The remainder of this paper is organized as follows. In Section~\ref{sec:problem}, we define our \ac{esc} problem formulation. In Section~\ref{sec:controller}, we present our \ac{esc} controller and prove its convergence. Finally, in Section~\ref{sec:examples} we demonstrate our \ac{esc} controller on five numerical example; one illustrative, one practical, and three benchmark. 

\subsubsection*{Notation:}
For a vector $v \in \reals^n$ and square matrix $M \in \reals^{n \times n}$, $\| v \|_M = \sqrt{v^\trans M v}$ is the weighted $2$-norm where the subscript is omitted for the identity matrix $\| v \| = \sqrt{v^\trans v}$. For a square matrix $M \in \reals^{n \times n}$, $\ubar{\lambda}(M)$ and $\bar{\lambda}(M)$ denote its smallest and largest eigenvalues respectively and $\| M \| = \sup\{ \| Mx \| : \|x \| \leq 1\}$ is the induced $2$-norm. A function $\alpha : [0,\infty) \rightarrow [0,\infty)$ is class-$\K$ if $\alpha(0) = 0$ and it is strictly increasing. A function $\alpha : [0,\infty) \rightarrow [0,\infty)$ is class-$\K_\infty$ if it is class-$\mathcal{K}$ and $\lim_{a\rightarrow \infty} \alpha(a) = \infty$. A function $\beta:[0,\infty)^2 \rightarrow [0,\infty)$ is class-$\mathcal{KL}$ if $\beta(\cdot,t)$ is class-$\mathcal{K}$ $\forall t>0$ and $\beta(r,\cdot)$ is continuous and strictly decreasing $\forall r>0$. A system $x_{t+1} = f(x_t,d_t)$ is input-to-state stable if $\| x_t \| \leq \beta(\| x_0\|,t) + \gamma(\sup_t \|d_t\|)$ where $\beta \in \KL$ and $\gamma \in \K$. A function $f$ is $\mathcal{C}^n$ if the derivatives derivatives $f^{(1)},\dots,f^{(n)}$ exist and are continuous. 

\section{ESC Problem Statement}
\label{sec:problem}

Consider the following discrete-time nonlinear system
\begin{subequations}
\label{eq:plant}
\begin{align}
x_{t+1} &= f\big(x_t,u_t\big) \label{eq:plant-dynamics} \\
y_t &= g\big(x_t,u_t\big) \label{eq:plant-output}
\end{align}
\end{subequations}
where $x_t \in \reals^{n_x}$ is the state, $u_t \in \reals^{n_r}$ is the control input, and $y_t \in \reals^{n_y}$ are measured output other than cost. We make the following assumptions about the plant~\eqref{eq:plant}.

\begin{assumption}[Plant]
\label{assume:plant}
$\null$
\begin{enumerate}[(a)]
\item\label{assume:plant-iss} The plant is  controllable, observable, and Lipschitz continuous. Each input $u$ corresponds a unique \ac{iss} equilibrium state $\pi(u)$ where $\pi$ is Lipschitz continuous. 
\item\label{assume:plant-tracking} The output $y_t$ tracks $y_t \rightarrow \bar{r}$ constant inputs $u_t = \bar{r}$.
\end{enumerate}
\end{assumption}

Assumption~\ref{assume:plant} is admittedly a strong assumption, but one that is consistent with many industrial applications where \ac{esc} is applied to a closed-loop system with a well-designed controller and heavy instrumentation. When satisfied, this assumption can be used to improve the performance of \ac{esc} controllers. Assumption~\ref{assume:plant}\ref{assume:plant-iss} means that the closed-loop system~\eqref{eq:plant} is robustly stable. Thus, bounded perturbation of the input $u_t$ cause bounded perturbation of the output $y_t$, allowing for safe exploration without risking instability. Assumption~\ref{assume:plant}\ref{assume:plant-iss} is consistent with the assumptions made in other \ac{esc} literature e.g.~\cite{ATTA2019147, Guay2015, guay2015time}.

Assumption~\ref{assume:plant}\ref{assume:plant-tracking} amounts to assuming that the steady-state map of the system is identity $g(\pi(u),u)= I$. This assumption is made for notational simplicity. The steady-state cost $\ell(u) = \cost(g(\pi(u),u))$ with respect to the input $u$ depends on the steady-state map $g(\pi(\cdot),\cdot)$. Without loss of generality, we can transform the inputs $u= \pi(r)$ to produce a plant~\eqref{eq:plant} whose steady-state map $g(g^{-1}(\cdot)) = I$ is identity where $g$ is invertible since both $h$ and $\pi$ are invertible. This simplifies the notation (but not the analysis) since $\nabla_r \cost(g(\pi(r))) = \nabla_y \cost$ instead of $\nabla_u \ell = \nabla_y \cost \nabla_u h + \nabla_y \cost \nabla_x h \nabla \pi$.

The \textbf{objective} of the \ac{esc} is to  find a operating condition $u_t = \ropt$ such that the plant~\eqref{eq:plant} optimizes an unknown steady-state cost $\cost(y)$. The optimal equilibrium is defined as
\begin{subequations}
\label{eq:equilibrium}
\begin{alignat}{3}
(\bar{y}^\star, \bar{u}^\star, \bar{x}^\star) =~ &&
\arg\min &~ \cost(\bar{y})\\
&&\mathrm{s.t.} &~ \bar{x} = f(\bar{x},\bar{u}) \\
&&&~ \bar{y} = g(\bar{x},\bar{u} ).
\end{alignat}
\end{subequations}
We make the following assumptions about the cost $\cost \in \mathcal{C}^2$.

\begin{assumption}[Cost]
\label{assume:cost}
$\null$
\begin{enumerate}[(a)]
\item\label{assume:cost-curvature} The cost $\cost \in \mathcal{C}^2$ has bounded curvature $\ubar{H} \preceq \nabla^2 \cost \preceq \bar{H}$. 
\item\label{assume:cost-lyapunov} The cost $\cost \in \mathcal{C}^2$ is bounded by class-$\K_\infty$ functions $\kappa_1( \| y- y^\star \|) \leq \cost(y) \leq \kappa_2( \| y- y^\star \|)$ and its gradient $\nabla \cost$ of the cost satisfies $\| \nabla \cost(y) \| \geq \kappa_3( \| y- y^\star \|)$ for some class-$\K_\infty$ function $\kappa_3$. 
\end{enumerate}
\end{assumption}

Assumption~\ref{assume:cost}\ref{assume:cost-curvature} will be used to bound the estimation errors of our gradient estimator. This assumption holds if and only if the cost gradient $\nabla \cost \in \mathcal{C}^1$ Lipschitz continuous i.e. $\ubar{H} = -hI$ and $\bar{H} = hI$ implies $\| \nabla \cost(y_1) - \nabla \cost(y_2) \| \leq h \| y_1 - y_2 \|$. However, our \ac{esc} algorithm can exploit more nuanced curvature bounds $\ubar{H},\bar{H}$, if available, to improve the convergence rate. Note that the bounds $\ubar{H}$ and $\bar{H}$ are not required to be positive definite matrices. Thus, we are not assuming that the cost $\cost$ is convex. Indeed, for two of our numerical examples, the cost will be non-convex. 

Assumption~\ref{assume:cost}\ref{assume:cost-lyapunov} means that driving the cost gradient to zero $\nabla \cost \rightarrow 0$ results in the output converging to the optimal $y_t \rightarrow y^\star$. This assumption will be used to prove the stability of the optimal equilibrium~\eqref{eq:equilibrium}. If the cost $\cost$ is convex (i.e. $0 \preceq \ubar{H} \preceq \nabla^2 \cost$) then Assumption~\ref{assume:cost}\ref{assume:cost-curvature} implies that Assumption~\ref{assume:cost}\ref{assume:cost-lyapunov} holds locally. However, this assumption can hold for non-convex costs $\cost$, like those we will consider in our numerical examples. 

\section{Adaptive Gradient ESC Algorithm}
\label{sec:controller}
\acresetall

Our \ac{esc} is given by 
\begin{subequations}
\label{eq:controller}
\begin{align}
r_{t+1} &= r_t + \begin{cases} 
-\alpha_t K \hat{\theta}_t & \text{ if } \alpha_t \geq \underline{\alpha} \\
0 & \text{ otherwise} 
\end{cases} \\
u_{t} &= r_t + d_t 
\end{align}
\end{subequations}
where the state $r_t$ of the controller is the current estimate of the optimal reference $\ropt$ and the input $u_t$ is the reference $r_t$ plus a dither signal $d_t$. The step-size $\alpha_t$ and controller gain $K$ will be described below. For $\alpha_t \geq \ubar{\alpha}$, the \ac{esc} controller~\eqref{eq:controller} is a discrete-time integral controller $r_{t+1} = r_{t} - \alpha_t K \hat{\theta}_t$. 

The gradient $\theta_t = \nabla \cost (r_t)$ of the cost function $\cost(r_t)$ at the current reference set-point $r_t$ is estimated by the following finite-horizon \ac{bls} estimator
\begin{subequations}
\label{eq:estimator}
\begin{align}
\Lambda_t^{-1} &= \frac{1}{N}\sum_{k=t-1}^{t-N}  w_k \Delta y_k \Delta y_k^\trans 
\label{eq:estimator-information} \\
\hat{\theta}_t &= \frac{\Lambda_t}{N} \sum_{k=t-1}^{t-N} 
w_k   \Delta y_k  \Big(\Delta \cost_k \!+\! \Delta y_k^\trans \hat{H} \big(e_t \!-\! \tfrac{1}{2} \Delta y_k\big)\Big)
\label{eq:estimator-mean}
\end{align}
where $\Delta y_k = y_t - y_k$ and $\Delta \cost_k = \cost(y_t) - \cost(y_k)$ are changes in the measurements of the output and cost, respectively, and the batch horizon $N \geq n_y$ is at least $n_y$. The existence of the inverse $\Lambda_t$ of the information matrix~\eqref{eq:estimator-information} requires that the output $y_t$ of the plant~\eqref{eq:plant} is persistently exciting, which is achieved using the dither $d_t$ in the controller~\eqref{eq:controller}. The correction term $\Delta y_k^\trans \hat{H} ( e_t + \tfrac{1}{2}\Delta y_k)$ compensates for the tracking error $e_t = y_t - r_t$ and transients $\Delta y_k \neq 0$ where $\hat{H} = \tfrac{1}{2} (\bar{H} + \ubar{H})$ is the \textit{median} curvature $\nabla^2 \cost$ of the cost $\cost$. Without the measurements of the outputs $y_t$, the \ac{esc} controller~\eqref{eq:controller} would need to be detuned to conservatively allow the plant~\eqref{eq:plant} to settle near the equilibrium $\pi(r_t)$. If only a Lipschitz bound $h$ on gradient $\nabla \cost$ is known, then the correction term disappears $\Delta y_k^\trans \hat{H} ( e_t + \tfrac{1}{2}\Delta y_k) = 0$ since $\hat{H} = 0$ when $\ubar{H} = -hI$ and $\bar{H} = hI$. For a convex cost with known Lipschitz bound $h$, the correction term is $\tfrac{h}{2}\Delta y_k^\trans ( e_t + \tfrac{1}{2}\Delta y_k)$. The \textbf{novel} weighting $w_k$ is given by
\begin{align}
\label{eq:weighting}
w_k = \frac{1}{\tfrac{1}{2}\| \Delta y_k \| \| \Delta y_k \|_{\tilde{H}}  \big( \| e_t \|_{\tilde{H}} + \tfrac{1}{2} \| \Delta y_k \|_{\tilde{H}} \big)}
\end{align}
\end{subequations}
where $\tilde{H} = \bar{H} - \ubar{H}$ bounds the range of curvature $\nabla^2 \cost$ of the unknown cost $\cost$. 
When the tracking errors $\|e_t\| \gg 1$ and output transients $\| \Delta y_k \|\gg 1$ are large, the weighting is small $w_k \ll 1$ indicating that the data-point $\{\Delta \cost_k, \Delta y_k\}$ will not provide reliable information about the steady-state gradient $\nabla \cost(r_t)$. When only a Lipschitz bound $h$ on the gradient $\nabla \cost$ is known, the weightings~\eqref{eq:weighting} simplify
\begin{align*}
w_k = \frac{1}{h^2 \| \Delta y_k \|^2 \big( 2 \| e_t \| + \| \Delta y_k \| \big)}.
\end{align*}
We will show that this weighting guarantees that the gradient estimation errors $\tilde{\theta}_t = \hat{\theta}_t - \nabla \cost(r_t)$ are bounded. 

Our \textbf{main contribution} is the adaptive step-size
\begin{align}
\label{eq:step-size}
\alpha_t = \max\left\{0, 1 - \frac{\big\| \Lambda_t^{\frac{1}{2}} K \hat{\theta}_t \big\|}{\| \hat{\theta}_t \|_K^2} \right\}
\end{align}
which dictates both the mode and the aggressiveness of the controller~\eqref{eq:controller} based on the quality of the gradient estimate $\hat{\theta}_t$. If the step-size~\eqref{eq:step-size} is small $\alpha_t < \underline{\alpha}$, then the controller~\eqref{eq:controller} enters the so called \textit{exploration mode} where the state $r_t$ of the controller~\eqref{eq:controller} remains constant while the dither signal $d_t$ probes the plant~\eqref{eq:plant} to improve the gradient estimate. If $\alpha_t \geq \ubar{\alpha}$ then the controller~\eqref{eq:controller} enters the so called \textit{exploitation mode} where it descends the estimated gradient $\hat{\theta}_t$ with $K = K^\trans \succ 0$. Furthermore, the aggressiveness of this descent is dictated by the step-size~\eqref{eq:step-size}. The $\max$ operator ensures that the step-size is non-negative and well-defined. If $\| \hat{\theta}_t \|_K^2 = 0$, then either we have perfectly misestimated the gradient $\tilde{\theta} = \nabla \cost$ or perfectly estimated a zero gradient $\hat{\theta} = \nabla \cost = 0$ (i.e. we are at optimal). In either case, the step-size~\eqref{eq:step-size} is zero since the controller should not step.  

To better understand the intuition behind the adaptive step-size~\eqref{eq:step-size}, consider the case where the controller gain $K$ and estimator covariance $\Lambda_t$ are balanced i.e. $K \approx k I$ and $\Lambda \approx \sigma^2 I$. Then, we can approximate the adaptive step-size~\eqref{eq:step-size} as
\begin{align*}
\alpha_t = \max\left\{ 0, 1 - \frac{\| \Lambda_t^{\frac{1}{2}} \| \| K \hat{\theta}_t\|}{\|  \hat{\theta} \| \| K \hat{\theta}_t \|} \right\} \approx  1 - \frac{\sigma}{\mu} .
\end{align*} 
where $\sigma / \mu$ is the \textit{noise-to-signal ratio} of the gradient estimate and $\mu = \| \hat{\theta} \|_K^2$ is the size of the descent direction. If the \textit{noise-to-signal ratio} is small $\sigma / \mu \ll 1$, then $\alpha_t \approx 1$, allowing the \ac{esc} controller~\eqref{eq:controller} to aggressively exploit the high-quality gradient estimate $\hat{\theta}_t$. Conversely, if the \textit{noise-to-signal ratio} is large $\sigma / \mu \approx 1$ then the reduced step-size $\alpha_t \ll 1$ slows the gradient descent. Thus, the adaptive step-size~\eqref{eq:step-size} provides a reactive separation of time-scales between the controller~\eqref{eq:controller} and estimator~\eqref{eq:estimator}. 

The positive definite controller gain $K = K^\trans \succ 0$ of the \ac{esc} controller~\eqref{eq:controller} must satisfy the matrix inequality
\begin{align}
\label{eq:esc-gain}
K -  K \big( \bar{H} + \gamma I \big) K \succeq 0
\end{align}
for some scalar $\gamma$. In Corollary~\ref{cor:linear-tuning} we will describe how to tuning of the controller gain~\eqref{eq:esc-gain} for a linear plant~\eqref{eq:plant}. If the plant~\eqref{eq:plant} has trivial dynamics, then the gain~\eqref{eq:esc-gain} is $K = \frac{1}{2} \bar{H}$, which is the ideal choice for the (non-dynamic) optimization problem~\eqref{eq:equilibrium}. For a dynamic plant, the gain~\eqref{eq:esc-gain} incorporates information about both the plant~\eqref{eq:plant} and optimization problem~\eqref{eq:equilibrium} to improve convergence and prevent instabilty. 

The following theorem proves that the \ac{esc} controller~\eqref{eq:controller} converges to the reference $\ropt$ that drives the plant~\eqref{eq:plant} to a neighborhood of the optimal equilibrium~\eqref{eq:equilibrium}. 

\begin{theorem}
\label{thm:main}
Let Assumptions~\ref{assume:plant} and~\ref{assume:cost} hold. Let the dither $d_t$ be persistently exciting and bounded $\| d_t \| \leq \delta$. Let $K$ satisfy~\eqref{eq:esc-gain}. Then the optimal equilibrium~\eqref{eq:equilibrium} is input-to-state stable for the closed-loop system~\eqref{eq:plant} and~\eqref{eq:controller}-\eqref{eq:step-size}. 
\end{theorem}

Theorem~\ref{thm:main} says that the \ac{esc} controller~\eqref{eq:controller} drives the plant~\eqref{eq:plant} to a neighborhood of the optimal equilibrium~\eqref{eq:equilibrium} where the size of this neighborhood depends on the amplitude $\delta$ of the dither $d_t$. In practice, a vanishing dither~\cite{ATTA2019147,yin2019optimizing} can be used to provide convergence to the optimal, rather than only a neighborhood.

\subsection{Proof of Theorem~\ref{thm:main}}

In this section, we prove Theorem~\ref{thm:main}. First, we analyze the \ac{esc} controller~\eqref{eq:controller} under the idealistic condition where the \ac{bls} estimator~\eqref{eq:estimator} is perfect $\hat{\theta}_t = \nabla \cost(r_t)$ and thus the step-size~\eqref{eq:step-size} is maximal $\alpha_t = 1$. We will then examine how the adaptive step-size $\alpha_t$ can be used to make the \ac{esc} controller~\eqref{eq:controller} robust to imperfect gradient estimates $\hat{\theta}_t \neq \nabla \cost(r_t)$. Finally, we will show that our \ac{bls} estimator~\eqref{eq:estimator} satisfies our conditions for stability. 

\begin{proposition}
\label{prop:convergence}
Let Assumptions~\ref{assume:plant} and~\ref{assume:cost} hold. Let $K$ satisfy~\eqref{eq:esc-gain}. Then the optimal equilibrium~\eqref{eq:equilibrium} is \ac{iss} for the closed-loop system~\eqref{eq:plant} and~\eqref{eq:controller} where $\hat{\theta}_t = \nabla \cost(r_t)$ and $\alpha_t = 1$. 
\end{proposition}

\begin{proof}
Define $\tilde{r}_t = r_t-\ropt$ and $\tilde{x}_t = x_t - \pi(r_t)$. We will prove input-to-state stability using a candidate Lyapunov function of the form
\begin{align}
\label{eq:lyapunov}
V(\tilde{x},\tilde{r}) = V_x(\tilde{x}) + V_r(\tilde{r})
\end{align}
where $V_x$ and $V_r$ are candidate Lyapunov functions for the plant~\eqref{eq:plant} and controller~\eqref{eq:controller}, respectively.

Since each constant equilibrium $\bar x^+=\bar x = \pi(\bar r)$ of the plant~\eqref{eq:plant} is \ac{iss} by Assumption~\ref{assume:plant}, there exists an \ac{iss} Lyapunov function $V_x$ that satisfies
\begin{subequations}
\label{eq:lyapunov-plant}
\begin{align}
\ubar{p}\big(\| \tilde{x} \|\big) \leq V_x\big( \tilde{x} \big) &\leq \phantom{-}\bar{p}\big(\| \tilde{x}\|\big) \\
V_x\big(f(x,\bar{u})-\bar{x}^+\big) - V_x\big( \tilde{x}\big) &\leq -q\big(\| \tilde{x} \|\big) + \sigma(\| d_t \|) 
\end{align}
\end{subequations}
according to the converse Lyapunov function theorem~\cite{Jiang2002} where $\ubar{p},~\bar{p},~q \in \K_\infty$ and $\sigma \in \K$. Here, the input $u_t = \bar r + d_t$ is dithered about the set-point $\bar r$. When the target equilibrium is varying $\bar{x}^+ \neq \bar{x}$, then the Lyapunov function~\eqref{eq:lyapunov-plant} satisfies
\begin{align*}
\Delta V_x(\tilde{x}) = &
V_x(\tilde{x}^+) - V_x(x^+ - \bar{x}) 
+
\underbrace{V_x(x^+ - \bar{x})  - V_x(\tilde{x})}_{\leq -q(\|\tilde{x}\|)}.
\end{align*}
where the first-term is the increase of the Lyapunov function due to the changing set-point and the second-term decrease due to stability. 
We can assume without loss of generality that $V_x$ is smooth~\cite{Kellett2003}, so that the first term above can be bounded by a quadratic 
\begin{align*}
\Delta V _x(\tilde{x})  \leq \nabla V(\tilde{x})^\trans \Delta\bar{x} + \tfrac{\rho}{2} \Delta\bar{x}^\trans \Delta\bar{x}
\end{align*}
where $\Delta \bar{x} = \bar{x}^+ - \bar{x}$ is the change in the equilibrium state and $\rho I \succeq \nabla^2 V_x$ is an upper-bound on the curvature of $V_x$. By Young's inequality $\nabla V^\trans \Delta \bar{x} \leq \tfrac{1}{2\gamma_0} \nabla V_x^\trans \nabla V_x + \tfrac{\gamma_0}{2} \Delta \bar{x}^\trans \Delta \bar{x}$, we have 
\begin{align}
\label{eq:peter-and-paul}
\Delta V _x(\tilde{x}) \leq  -q + \frac{1}{2\gamma_0}  \nabla V_x^\trans P^{-1} \nabla V_x  + \frac{\rho+\gamma_0}{2} \Delta \bar{x}^\trans P \Delta \bar{x}
\end{align}
where  $q_x^0 = q - \tfrac{1}{2\gamma_0} \nabla V_x^\trans \nabla V_x \succ 0$ for an appropriate scaling of $\gamma_0>0$. Since $\pi$ is Lipschitz continuous $\| \Delta \bar{x} \| \leq \ell_\pi \| \Delta r \|$, we obtain
$$\Delta V_x(\tilde{x}) \leq -q_x^0(\| \tilde{x} \|) + \tfrac{1}{2}  \gamma \hat{\theta}^\trans K^2 \hat{\theta}$$
where $\gamma = (\rho+\gamma_0 ) \ell_\pi^2$ and $\Delta r = K \hat{\theta}$. 

A natural choice for the controller~\eqref{eq:controller} Lyapunov function $V_r$ is the cost function
\begin{align}
\label{eq:lyapunov-controller}
V_r(\tilde{r}) = \cost(\tilde{r}+\ropt) - \cost(\ropt)
\end{align}
where $V_r(0) = 0$ by construction. By Assumption~\ref{assume:cost}, the cost $\cost$ is bounded above and below by class-$\K_\infty$ functions. By Taylor's theorem\footnote{Taylor's theorem, not a Taylor approximation.} and the controller dynamics~\eqref{eq:controller}, we have 
\begin{align*}
\Delta V_r =  V_r(\tilde{r}^+) - V_r(\tilde{r}) 
&\leq \nabla\cost(r)^\trans\Delta r + \tfrac{1}{2} \Delta r^\trans \bar{H} \Delta r \\
&= -\hat{\theta}^\trans K \hat{\theta}  + \tfrac{1}{2} \hat{\theta}^\trans K \bar{H} K \hat{\theta} 
\end{align*}
where $\hat{\theta} = \nabla \cost$ and $\Delta r = K \hat{\theta}$ for $\alpha_t = 1$. Thus, the combined Lyapunov function~\eqref{eq:lyapunov} satisfies 
\begin{align*}
\Delta V 
&\leq -q_x^0(\| \tilde{x} \|) - \hat{\theta}^\trans K \hat{\theta} + \tfrac{1}{2} \hat{\theta}^\trans K ( \gamma I + \bar{H})K \hat{\theta} \\
&\leq -q_x^0(\| \tilde{x} \|) - \tfrac{1}{2}  \hat{\theta}^\trans K \hat{\theta}
\end{align*}
where $K - \tfrac{1}{2} K \big(\gamma I + \bar{H} \big) K \succeq \tfrac{1}{2} K$ by~\eqref{eq:esc-gain}. By Assumption~\ref{assume:cost}, the state $\tilde{r}$ of the controller~\eqref{eq:controller} is bounded by the gradient $\kappa_3( \| \tilde{r} \|) \leq \| \theta \|$. Therefore, the combined Lyapunov function~\eqref{eq:lyapunov} is bounded by class-$\mathcal{K}_\infty$ functions and satisfies
$$\Delta V(\tilde{x},\tilde{r}) \leq - q_x^0(\|\tilde{x}\|) - q_r^0(\|\tilde{r}\|) + \sigma(\| d_t \|) $$
where $q_x^0 = q - \tfrac{1}{2\gamma_0} \nabla V_x^\trans \nabla V_x$ and $q_r^0( \tilde{r} ) = \ubar{\lambda}(K) \kappa_3( \| \tilde{r} \| )$ are class-$\mathcal{K}_\infty$ functions and $\sigma \in \K$. Thus, by Proposition~2.3 in~\cite{Jiang2001} the optimal equilibrium~\eqref{eq:equilibrium} is \ac{iss}. 
\end{proof}

The Lyapunov function~\eqref{eq:lyapunov} defined in the proof of Proposition~\ref{prop:convergence} will be used to prove Theorem~\ref{thm:main}. The proof of Proposition~\ref{prop:convergence} uses similar arguments to the proof of Theorem 4.1 from~\cite{Guay2015}. However, our proof highlights the issue that without an appropriate controller gain~\eqref{eq:esc-gain}, the integral-action of the gradient descent can destabilize the plant~\eqref{eq:plant}, even when the gradient is perfectly estimated $\hat{\theta} = \nabla \cost$. As a quick aside, the following corollary shows how the derivation of the controller gain~\eqref{eq:esc-gain} for linear plants. 

\begin{corollary}
\label{cor:linear-tuning}
For a linear plant~\eqref{eq:plant}, Proposition~\ref{prop:convergence} holds if the \ac{esc} controller gain $K$ satisfies
\begin{align}
K - K \big( \bar{H} \!+\! B^\trans (I \!-\! A)^{-\trans} P Q^{-1} P (I \!-\! A)^\trans B \big) K \succeq 0.
\end{align}
\end{corollary}

\begin{proof}
For a linear plant~\eqref{eq:plant}, we can use a quadratic $V_x(\tilde{x}) = \tfrac{1}{2}\tilde{x}^\trans P \tilde{x}$ plant Lyapunov function~\eqref{eq:lyapunov-plant} where $P$ satisfies the Lyapunov equation $A^\trans P A - P = -Q$ for some $Q \succ 0$. We can then use a matrix $\Gamma$, instead of a scalar $\gamma_0$, in Young's inequality~\eqref{eq:peter-and-paul} to obtain
\begin{align*}
\Delta V_x(\tilde{x}) = - \tilde{x}^\trans Q \tilde{x} + \tfrac{1}{2} \tilde{x}^\trans P \Gamma P \tilde{x} + \tfrac{1}{2} \Delta \bar{x}^\trans ( P + \Gamma^{-1}) \Delta \bar{x}^\trans
\end{align*}
for some $\Gamma = \Gamma^\trans \succ 0$ where $\nabla V_x = P\tilde{x}$ and $\nabla^2 V_x = P$. Following the argument of the proof of Proposition~\ref{prop:convergence},  we require $-Q + P\Gamma P \preceq 0$ for stability. Or equivalently $\Gamma \preceq P^{-1} Q P^{-1}$. Thus, 
\begin{align*}
\Delta V_x(\tilde{x}) \preceq \tfrac{1}{2} \Delta \bar{x}^\trans ( P + PQ^{-1}P) \Delta \bar{x}.
\end{align*}
Finally, note that for a linear plant~\eqref{eq:plant} the change in equilibrium state $\bar{x}$ satisfies $\Delta \bar{x} = (I-A)^{-1} B \Delta r = (I-A)^{-1} BK \hat{\theta}$. The remainder of the stability proof is identical to the proof of Proposition~\ref{prop:convergence}.
\end{proof}

The linear tuning in Corollary~\ref{cor:linear-tuning} can provide some insight for tuning the controller~\eqref{eq:controller} gain~\eqref{eq:esc-gain} for a nonlinear plants~\eqref{eq:plant}, which is often challenging. 

Next, we examine the robustness of the \ac{esc} controller~\eqref{eq:controller} to imperfect gradient estimates $\hat{\theta}_t \neq \nabla \cost(r_t)$.  We will consider  gradient estimation errors $\tilde{\theta}_t := \hat{\theta}_t - \nabla \cost(r_t)$ that are slightly smaller $\| \tilde{\theta} \|_K \leq (1-\ubar{\alpha}) \| \hat{\theta} \|_K$ than the estimated gradient $\hat{\theta}_t$ 
\begin{align}
\label{eq:robust-set}
\tilde{\Theta}_t^{\max} = \left\{ \tilde{\theta} \in \reals^{n_y} : \| \tilde{\theta} \|_K \leq ( 1 - \ubar{\alpha}) \| \hat{\theta}_t \|_K  \right\}
\end{align}
where $0<\ubar{\alpha}\ll 1$. The following corollary shows that the optimal equilibrium~\eqref{eq:equilibrium} remains \ac{iss} for gradient estimation errors $\tilde{\theta}_t = \hat{\theta}_t - \nabla\cost(r_t)$ that satisfy the bound $\tilde{\theta}_t \in \tilde{\Theta}_t \subseteq  \tilde{\Theta}_t^{\max}$.

\begin{corollary}
\label{cor:exploitation}
Let Assumptions~\ref{assume:plant} and~\ref{assume:cost} hold. Let the set $\tilde{\Theta}_t$ of gradient estimation errors $\tilde{\theta}_t = \hat{\theta}_t - \nabla \cost_t \in \tilde{\Theta}_t$ satisfy the bound $\tilde{\Theta}_t \subseteq \tilde{\Theta}_t^{\max}$. Let $K$ satisfy~\eqref{eq:esc-gain} and the step-size $\alpha_t$ satisfy 
\begin{align}
\label{eq:step-size-generic}
\alpha_t^\star = \max\left\{ 0, 1 - \max_{\tilde{\theta} \in \tilde{\Theta}_t}  \frac{\tilde{\theta}^\trans K \hat{\theta}_t}{ \| \hat{\theta}_t\|_K^2}\right\}.
\end{align}
Then the optimal equilibrium~\eqref{eq:equilibrium} is \ac{iss} for the closed-loop system ~\eqref{eq:plant} and~\eqref{eq:controller}
\begin{align}
\label{eq:lyapunov-exploitation}
\Delta V(\tilde{x},\tilde{r}) \leq - q_x^1(\| \tilde{x} \|) - q_r^2(\| \tilde{r} \|) + \sigma(\| d_t \|)
\end{align}
where  $q_x^1, q_r^2 \in \K_\infty$ and $\sigma \in \K$.
\end{corollary}

\begin{proof}
For imperfect gradient estimates $\hat{\theta} \neq \nabla\!\cost$, the Lyapunov function~\eqref{eq:lyapunov} from Proposition~\ref{prop:convergence} satisfies 
\begin{align}
\label{pf:lyapunov-exploitation}
\Delta V \leq -q_x^0 - \alpha_t  \nabla\!\cost^\trans K\hat{\theta} + \tfrac{1}{2} \alpha_t^2 \hat{\theta}^\trans K \hat{\theta} + \sigma(\| d_t \|)
\end{align}
where the step-size $\alpha_t \geq 0$ is a design variable we can choose to promote stability while the worst-case gradient estimation error $\tilde{\theta}_t \in \tilde{\Theta}_t$ will try to prevent stability. Since the worst-case estimation error $\tilde{\theta}_t \in \tilde{\Theta}_t$ will depend on our choice of step-size $\alpha_t$, we have the following two-player zero-sum game
\begin{align}
\label{eq:game}
\min_{\alpha_t \geq 0}~\max_{\tilde{\theta}_t \in \tilde{\Theta}_t}~ -q_x^0 - \alpha_t (\hat{\theta} - \tilde{\theta})^\trans  K\hat{\theta} + \tfrac{1}{2} \alpha_t^2 \hat{\theta}^\trans K \hat{\theta} + \sigma(\| d_t \|)
\end{align}
where $\nabla\!\cost = \hat{\theta} - \tilde{\theta}$. Here, the adversary has the advantageous position of selecting the worst-case gradient estimation error $\tilde{\theta}_t \in \tilde{\Theta}_t$ based on our choice of step-size $\alpha_t \geq 0$. However, for this particular game~\eqref{eq:game} the optimal strategy $\tilde{\theta}_t^\star (\alpha_t)$ for the gradient estimation error happens to be independent of the step-size since 
\begin{align*}
\tilde{\theta}_t^\star 
&= \arg \max_{\tilde{\theta}_t \in \tilde{\Theta}_t}~ \alpha_t \hat{\theta}_t^\trans K \tilde{\theta}
\end{align*}
where the non-negative step-size $\alpha_t \geq 0$ only scales the linear cost $\hat{\theta}^\trans K$, but does not change its direction. The optimal strategy for the step-size $\alpha_t^\star$ is given by the following scalar quadratic program
\begin{align*}
\alpha_t^\star 
&= \arg\min_{\alpha_t \geq 0}~ -q_x^0 - \alpha_t (\hat{\theta} - \tilde{\theta}^\star(\alpha_t))^\trans  K\hat{\theta} + \tfrac{1}{2} \alpha_t^2 \hat{\theta}^\trans K \hat{\theta} \\
&= \arg\min_{\alpha_t \geq 0}~ \alpha_t \big( \hat{\theta}^\trans  K \tilde{\theta}^\star - \| \hat{\theta} \|_K^2 \big) + \tfrac{1}{2} \alpha_t^2 \| \hat{\theta} \|_K^2 
\end{align*}
where $\tilde{\theta}^\star(\alpha_t) = \max_{\tilde{\theta} \in \tilde{\Theta}_t}  \tilde{\theta}^\trans K \hat{\theta}_t$ is independent of $\alpha_t$. The step-size~\eqref{eq:step-size-generic} is the explicit optimal solution of this parametric quadratic program where $\tilde{\theta}_t^\star$ is the parameter. Thus,~\eqref{eq:step-size-generic} is the game-theoretic optimal step-size $\alpha_t$ for bounded gradient estimation errors $\tilde{\theta}_t \in \tilde{\Theta}_t$. The resulting change~\eqref{pf:lyapunov-exploitation} of the Lyapunov function~\eqref{eq:lyapunov} satisfies 
\begin{align*}
\Delta V = \Delta V_x + \Delta V_r  \leq -q_x^0 - (\alpha_t^\star)^2  \| \hat{\theta} \|_K^2 + \sigma(\| d_t \|). 
\end{align*}
To prove \ac{iss}, we will show that $(\alpha_t^\star)^2  \| \hat{\theta} \|_K^2$ is bounded by a class-$\K_\infty$ function $q_r^1(\| \tilde{r} \|)$ of the controller~\eqref{eq:controller} state $\tilde{r}$. Since $\tilde{\Theta}_t \subseteq \tilde{\Theta}_t^{\max}$, we have 
\begin{align*}
\alpha_t^\star  
&\geq   1 - \max_{\tilde{\theta} \in \tilde{\Theta}_t} \frac{\tilde{\theta}^\trans K \hat{\theta}_t}{ \| \hat{\theta}_t \|_K^2} 
\geq   1 - \max_{\tilde{\theta} \in \tilde{\Theta}_t} \frac{\| \tilde{\theta}_t \|_K \| \hat{\theta}_t \|_K}{ \| \hat{\theta}_t \|_K^2} = \ubar{\alpha}
\end{align*}
where the first inequality follows from the step-size~\eqref{eq:step-size-generic}, the second inequality is the Cauchy-Schwarz inequality, and the last inequality follows from the definition~\eqref{eq:robust-set} of the set $\tilde{\Theta}_t^{\max}$. Thus, $$(\alpha_t^\star)^2  \| \hat{\theta} \|_K^2 \geq \ubar{\alpha}^2  \| \hat{\theta} \|_K^2$$ where $\ubar{\alpha} >0$. Next, we need to bound the norm $\| \hat{\theta} \|_K^2$ of the estimated gradient $\hat{\theta}_t$ by the norm $\| \theta \|_K^2$ of the actual gradient $\theta_t = \nabla \cost(r_t)$. Since $\theta_t = \hat{\theta}_t - \tilde{\theta}_t$, we have 
\begin{align*}
\| \theta \|_K = \| \hat{\theta} - \tilde{\theta} \|_K 
&\leq \| \hat{\theta} \|_K + \| \tilde{\theta} \|_K \\
&\leq \| \hat{\theta} \|_K + (1-\ubar{\alpha}) \| \hat{\theta} \|_K \leq 2 \| \hat{\theta} \|_K
\end{align*}
where the first inequality is the triangle inequality and the second inequality follows from the definition~\eqref{eq:robust-set} of the set $\tilde{\Theta}_t^{\max}$. Rearranging terms and squaring yields $\ubar{\alpha}^2 \| \hat{\theta} \|_K^2  \geq \tfrac{1}{4}\ubar{\alpha}^2 \| \theta \|_K^2$. Finally, recall from the proof of Proposition~\ref{prop:convergence} that $$\| \theta \|_K^2 \geq \ubar{\lambda}(K) \kappa_3(\| \tilde{r} \|),$$ where $\kappa_3(\| \tilde{r} \|) \in \K_\infty$.  Therefore, we conclude that~\eqref{eq:lyapunov-exploitation} holds where $q_x^1(\| \tilde{x} \|) = q_x^0(\| \tilde{x} \|) \in \K_\infty$ and $$q_r^2(\| \tilde{r} \|) = \tfrac{1}{4}\ubar{\alpha}^2 \ubar{\lambda}(K) \kappa_3(\| \tilde{r} \|) \in \K_\infty.$$ 
\end{proof}

Corollary~\ref{cor:exploitation} shows that the adaptive step-size~\eqref{eq:step-size-generic} makes our \ac{esc} controller~\eqref{eq:controller} robust to bounded ~\eqref{eq:robust-set} gradient estimation errors $\tilde{\theta}_t = \hat{\theta}_t - \nabla \cost_t \in \tilde{\Theta}_t$. The adaptive step-size~\eqref{eq:step-size-generic} ensures that the Lyapunov function~\eqref{eq:lyapunov} decreases~\eqref{eq:lyapunov-exploitation} for the worst-case gradient estimation errors $\tilde{\theta}_t \in \tilde{\Theta}_t$. When the bounds $\tilde{\Theta}_t$ on the gradient estimation errors $\tilde{\theta}_t \in \tilde{\Theta}_t$ are tight, the step-size~\eqref{eq:step-size-generic} is large and the controller~\eqref{eq:controller} is aggressive. Interestingly, the worst-case gradient estimation error $\tilde{\theta}_t^\star$ does not try to alter the descent-direction $K \hat{\theta}$, but rather amplifies the descent-direction $\hat{\theta} = \frac{1}{\ubar{\alpha}} \theta$. This can cause the integral-action of the \ac{esc} controller~\eqref{eq:controller} to overshoot the optimal, potentially leading to unstable oscillation. 

Later in Lemma~\ref{lemma:estimation-error}, we will derive the specific bounds $\tilde{\Theta}_t^{\text{\sc bls}}$ on the gradient estimation errors produce by the \ac{bls} estimator~\eqref{eq:estimator}. We will also connect the game-theoretic step-size~\eqref{eq:step-size-generic} with the step-size~\eqref{eq:step-size} used in the \ac{esc} controller~\eqref{eq:controller} in Corollary~\ref{cor:step-size}. 

The set~\eqref{eq:robust-set} is the largest set $\tilde{\Theta}_t^{\max}$ of gradient estimation errors $\tilde{\theta}$ for which the \ac{esc} controller~\eqref{eq:controller} can decrease the Lyapunov function~\eqref{eq:lyapunov}. Thus, the controller remains in the exploration mode i.e. $\alpha_t^\star \geq \ubar{\alpha}$. When the bounds $\tilde{\Theta}_t$ on the estimation errors are too large $\tilde{\Theta}_t \not\subseteq \tilde{\Theta}_t^{\max}$, the \ac{esc} controller~\eqref{eq:controller} enters the exploration mode i.e. $\alpha_t^\star = 0$. The following trivial corollary shows that optimal equilibrium~\eqref{eq:equilibrium} remains stable (but not \ac{iss} nor asymptotically stable) when the \ac{esc} controller~\eqref{eq:controller} is in exploration mode $\alpha_t^\star = 0$. 

\begin{corollary}
\label{cor:exploration}
Let Assumptions~\ref{assume:plant} and~\ref{assume:cost} hold. Let $K$ satisfy~\eqref{eq:esc-gain} and the step-size satisfy $\alpha_t^\star = 0$. Then the plant~\eqref{eq:plant} is \ac{iss} with respect to the dither signal $d_t$
\begin{subequations}
\label{eq:lyapunov-exploration}
\begin{align}
\label{eq:lyapunov-exploration-plant}
\Delta V_x(\tilde{x}) \leq - q_x^2(\| \tilde{x} \|) + \sigma(\| d_t \|),
\end{align}
and the controller~\eqref{eq:controller} is stable
\begin{align}
\label{eq:lyapunov-exploration-control}
\Delta V_r(\tilde{r}) =0,
\end{align}
\end{subequations}
where $q_x^2 \in \K_\infty$ and $\sigma \in \K$. 
\end{corollary}

\begin{proof}
Since the controller~\eqref{eq:controller} state $r_{t+1} = r_t$ is constant in the exploration mode, the controller Lyapunov function is constant~\eqref{eq:lyapunov-exploration-control}. The decrease condition~\eqref{eq:lyapunov-exploration-plant} for the plant Lyapunov function follows from the \ac{iss} assumption~\eqref{eq:lyapunov-plant}. 
\end{proof}

Corollary~\ref{cor:exploration} shows that the divergence from the optimal equilibrium~\eqref{eq:equilibrium} is bounded for bounded dithers $d_t$. This allows the dither $d_t$ to safely probe the system~\eqref{eq:plant} to gather exciting data and reduce the gradient estimation errors $\tilde{\Theta}_t$. However, it does not provide convergence towards the optimal equilibrium~\eqref{eq:equilibrium}. For convergence, we need the \ac{esc} controller~\eqref{eq:controller} to interminably enter the exploitation mode $\alpha_t^\star \geq \ubar{\alpha}$. This requires that the bounds $\tilde{\Theta}_t^{\text{\sc bls}}$ on the gradient estimation errors produce by the \ac{bls} estimator~\eqref{eq:estimator} are sufficient small~\eqref{eq:robust-set}. The following lemma establishes bounds $\tilde{\Theta}_t^{\text{\sc bls}}$ on the gradient estimation errors produce by the \ac{bls} estimator~\eqref{eq:estimator}.

\begin{lemma}
\label{lemma:estimation-error}
Let Assumptions~\ref{assume:plant} and~\ref{assume:cost} hold. The estimation errors $\tilde{\theta}_t = \hat{\theta}_t - \nabla \cost_t$ produced by the \ac{bls} estimator~\eqref{eq:estimator} are contained in the set
\begin{align}
\label{eq:estimation-error}
\tilde{\Theta}_t^{\text{\sc bls}} = \left\{ \tilde{\theta} : \big\| \Lambda_t^{-1} \tilde{\theta} \big\|^2 \leq 1 \right\}.
\end{align}
\end{lemma}

\begin{proof}
To derive the set-bound~\eqref{eq:estimation-error} on the gradient estimation error $\tilde{\theta}_t$, we will express the gradient $\theta = \nabla \cost$ as a linear regression $\Delta \cost_k = \Delta y_k^\trans \theta + \omega_k$ where the \textit{noise} $\omega_k$ is set bounded, rather than stochastic.

According to Taylor's theorem\footnote{Note that we use Taylor's theorem, which is exact, not a Taylor approximation, which is an approximation.}, the cost $\cost$ satisfies 
\begin{align}
\label{eq:cost-taylor}
\Delta \cost(y_k) = \Delta y_k^\trans \nabla \cost(y_t)  + \tfrac{1}{2} \Delta y_k^\trans \nabla^2 \cost(z_1)  \Delta y_k
\end{align}
where the unknown curvature $H_1 = \nabla^2 \cost(z_1)$ is evaluated at an unknown point $z_1 = \mu_1 y_t + (1 - \mu_1 ) y_k$ for some $\mu_1 \in [0,1]$. The desired gradient $\nabla \cost(r_t)$ evaluated at the current reference $r_t$ is related to the gradient $\nabla \cost(y_t)$ evaluated at the current output $y_t$ by the mean-value theorem
\begin{align}
\label{eq:gradient-mvt}
\theta_t = \nabla\cost(r_t) = \nabla \cost(y_t) - \nabla^2 \cost(z_2)^\trans (r_t - y_t)
\end{align}
where $e_t = r_t - y_t$ is the tracking error and the unknown curvature $H_2 = \nabla^2 \cost(z_2)$ is again evaluated at an unknown point $z_2 = \mu_2 y_t + (1 - \mu_2 ) r_t$ for some $\mu_2 \in [0,1]$. Combining the non-approximations~\eqref{eq:cost-taylor} and~\eqref{eq:gradient-mvt}, yields
\begin{align}
\label{pf-cost-approx}
\Delta \cost_k &= \Delta y_k^\trans \theta_t + \underbrace{\tfrac{1}{2} \Delta y_k^\trans H_1 \Delta y_k}_{\omega_1} - \underbrace{e_t^\trans H_2 \Delta y_k}_{\omega_2} 
\end{align}
where $\theta_t = \nabla\cost(r_t)$, $\Delta \cost_k = \cost(y_k) - \cost(y_t)$ and $\Delta y_k = y_k - y_t$. The non-approximation~\eqref{pf-cost-approx} says that the cost $\cost$ can be written as a linear regression $\Delta \cost_k = \Delta y_k^\trans \theta + \omega_k$ where the unknown error term $\omega = \omega_1 - \omega_2$ accounts for the nonlinearity. We will show that the error term $\omega$ is bounded since the curvature $\nabla^2 \cost$ of the cost $\cost$ is bounded $\ubar{H} \preceq \nabla^2 \cost \preceq \bar{H}$. 

From the definition of positive definite matrices, the quadratic-form $\omega_1 = \tfrac{1}{2}\Delta y_k^\trans H_1 \Delta y_k$ is contained in the line interval 
\begin{align*}
\Omega_k^1 
&= \Big\{ \omega_1 \in \reals : \tfrac{1}{2}\Delta y_k^\trans \ubar{H} \Delta y_k \leq \omega_1 \leq \tfrac{1}{2}\Delta y_k^\trans \bar{H} \Delta y_k \Big\} \\
&=  \tfrac{1}{2}\| \Delta y_k \|_{\hat{H}}^2 + \tfrac{1}{4} \| \Delta y_k\|_{\tilde{H}}^2 \Big[-1,1\Big]
\end{align*}
where $\hat{H} = \tfrac{1}{2} (\bar{H} + \ubar{H})$ is the median curvature and $\tilde{H} = \bar{H} - \ubar{H}$ is the range of curvature. Deriving the bounds $\Omega_k^2$ on the nonlinearity $\omega_2 = e_t^\trans H_2 \Delta y_k$ is more complicated since it is not a quadratic-form. Since the linear function $f(H) = e_t^\trans H \Delta y_k$ is continuous, the image $\Omega_k^2 = f(\mathcal{H})$ of the connected set $\mathcal{H} = \{ H : \ubar{H} \preceq H \preceq \bar{H} \}$ is connected. Furthermore, since $\omega_2 \in \reals$ is a scalar, this set $\Omega_k^2$ is a line interval, specifically
\begin{align}
\label{eq:uncertainty-sdp}
\Omega_k^2 &= 
\Bigg\{ \omega_2 :
 \min_{\ubar{H} \preceq H \preceq \bar{H}} \tfrac{1}{2} \trace(CH)  \leq \omega_2 \leq \max_{\ubar{H} \preceq H \preceq \bar{H}}  \tfrac{1}{2} \trace(CH) 
 \Bigg\} 
\end{align}
where $$C = \Delta y_k e_t^\trans + e_t \Delta y_k^\trans$$ is the symmetric cost matrix. In other words, we can find the bounds on $\omega_2$ by solving two semi-definite programs~\eqref{eq:uncertainty-sdp} for the lower and upper bounds on the line interval. According to Theorem 2.2 from~\cite{Xu2011}, these semi-definite programs~\eqref{eq:uncertainty-sdp} have a closed-from solution, specifically
\begin{align*}
\min_{\ubar{H} \preceq H \preceq \bar{H}} \tfrac{1}{2} \trace(CH) = \tfrac{1}{2} \trace^{-}\big(\tilde{H}^{\frac{1}{2}}C \tilde{H}^{\frac{1}{2}} \big) + \tfrac{1}{2} \trace\big(C\ubar{H}\big)
\end{align*}
where $\trace^{-}$ is the trace of the projection of a matrix into the negative semi-definite cone i.e. the sum of its negative eigenvalues. The rank-2 matrix $\tilde{H}^{\frac{1}{2}}C \tilde{H}^{\frac{1}{2}}$ has exactly one negative eigenvalue $$\lambda_{-} = e_t^\trans \tilde{H} \Delta y_k - \| \Delta y_k \|_{\tilde{H}} \| e_t \|_{\tilde{H}}$$. Thus, 
\begin{align*}
\min_{\ubar{H} \preceq H \preceq \bar{H}} \tfrac{1}{2} \trace(CH) 
&= \tfrac{1}{2} \Delta y_k^\trans \tilde{H} e_t  \!-\! \tfrac{1}{2} \| \Delta y_k \|_{\tilde{H}} \| e_t \|_{\tilde{H}} \!+\! e_t^\trans \ubar{H} \Delta y_k \\
&= \Delta y_k^\trans \hat{H} e_t \!-\! \tfrac{1}{2} \| \Delta y_k \|_{\tilde{H}} \| e_t \|_{\tilde{H}}.
\end{align*}
Similarly, we can obtain the upper-bound $$\omega_2 \leq \Delta y_k^\trans \hat{H} e_t + \| \Delta y_k \|_{\tilde{H}} \| e_t \|_{\tilde{H}}.$$ 
Thus, 
\begin{align*}
\Omega_k^2 &= 
\Delta y_k^\trans \hat{H} e_t + \tfrac{1}{2} \| \Delta y_k \|_{\tilde{H}} \| e_t \|_{\tilde{H}} \Big[-1,1\Big]
\end{align*}
Therefore, the total error $\omega = \omega_1 - \omega_2$ is contained in the line interval 
\begin{align*}
\Omega_k &= 
- \Delta y_k^\trans \hat{H} \big( e_t - \tfrac{1}{2} \Delta y_k \big) + 
\frac{1}{w_k \| \Delta y\|}
\Big[-1,1\Big].
\end{align*}
where the weighting $w_k$ was defined in~\eqref{eq:weighting}. Thus, cost~\eqref{pf-cost-approx} can be rewritten  
\begin{align}
\label{pf:cost}
\Delta \cost_k + \Delta y_k^\trans \hat{H} \big( e_t - \tfrac{1}{2} \Delta y_k \big) = \Delta y_k^\trans \theta_t + \frac{1}{w_k \| \Delta y_k \|} \nu_k 
\end{align}
where $\nu_k \in [-1,1]$ is the normalized \textit{noise}. Substituting~\eqref{pf:cost} into the \ac{bls}~\eqref{eq:estimator} yields
\begin{align*}
\hat{\theta}_t 
= \Lambda_t \underbrace{ \left( \frac{1}{N} \sum_{k=t-1}^{t-N} w_k \Delta y_k \Delta y_k^\trans \right)}_{\Lambda_t^{-1}} \theta_t +  
\Lambda_t  \frac{1}{N} \sum_{k=t-1}^{t-N} \frac{\Delta y_k }{ \| \Delta y_k \|} \nu_k. \notag
\end{align*}
Thus, by definition of the estimation error $\tilde{\theta}_t = \hat{\theta}_t - \theta_1$, the bound~\eqref{eq:estimation-error} holds since
\begin{align*}
 \big\| \Lambda_t^{-1} \tilde{\theta}_t \big\| =
 \left\| \frac{1}{N}  \displaystyle{\sum}_{k=t-1}^{t-N} \frac{\Delta y_k}{\| \Delta y_k \|} \nu_k \right\| \leq  \frac{1}{N}  \displaystyle{\sum}_{k=t-1}^{t-N}  \big| \nu_k \big| \leq 1
\end{align*}
where the first inequality is the triangle inequality and the second inequality follows from $| \nu_k | \leq 1$ since $\nu_k \in [-1,1] \subset \reals$. 
\end{proof}

Lemma~\ref{lemma:estimation-error} describes a (possibly degenerate) ellipsoidal set $\tilde{\Theta}_t^{\text{\sc bls}}$ that bounds the gradient estimation errors produce by the \ac{bls} estimator~\eqref{eq:estimator}. The following lemma shows that for persistently exciting and local data $\{ \cost_k, y_k \}_{t=k}^t$, the ellipsoid~\eqref{eq:estimation-error} is non-degenerate and eventually satisfies the bounds $\tilde{\Theta}_t^{\text{\sc bls}} \subseteq \tilde{\Theta}_t^{\max}$, allowing the \ac{esc} controller~\eqref{eq:controller} to return to the exploitation mode. 

\begin{lemma}
\label{lemma:time-bound}
Let Assumptions~\ref{assume:plant} and~\ref{assume:cost} hold. Let $\nabla \cost(r_t) \neq 0$. Let the dither $d_t$ be persistently exciting and bounded $\| d_t \| \leq \delta$. Then there exists a finite time $T<\infty$ and non-zero dither amplitude $\delta > 0$ such that~\eqref{eq:estimation-error} satisfies $\tilde{\Theta}_{t+T}^{\text{\sc bls}} \subseteq \tilde{\Theta}_{t+T}^{\max}$.
\end{lemma}

\begin{proof}
For notational simplicity, we will drop the time indices. 

First, we will transform the desired condition $\tilde{\Theta}^{\text{\sc bls}} \subseteq \tilde{\Theta}^{\max}$ into a more readily verifiable form. From the definition~\eqref{eq:robust-set} of $\tilde{\Theta}^{\max}$, the condition $\tilde{\Theta}^{\text{\sc bls}} \subseteq \tilde{\Theta}^{\max}$ holds if and only if $\| \tilde{\theta} \|_K \leq (1-\ubar{\alpha}) \| \hat{\theta} \|_K$ for all $\tilde{\theta} \in \tilde{\Theta}^{\text{\sc bls}}$. Substituting $\hat{\theta} = \theta + \tilde{\theta}$ and expanding the norm $\| \hat{\theta} \|_K^2 = \| \theta + \tilde{\theta} \|_K^2$, we have the following equivalent condition
\begin{align*}
\ubar{\alpha} (2 - \ubar{\alpha}) \| \tilde{\theta} \|_K^2 - 2(1-\ubar{\alpha})^2 \theta^\trans K \tilde{\theta} - (1-\ubar{\alpha})^2 \| \theta \|_K^2 \leq 0
\end{align*}
for all $\tilde{\theta} \in \tilde{\Theta}^{\text{\sc bls}}$ where $1 - (1-\ubar{\alpha})^2 = \ubar{\alpha} (2 - \ubar{\alpha})$. Using the Cauchy-Schwarz inequality $-\theta^\trans K \tilde{\theta} \leq \| \theta \|_K \| \tilde{\theta} \|_K$, we obtain the following conservative condition 
\begin{align*}
\ubar{\alpha} (2 - \ubar{\alpha}) \| \tilde{\theta} \|_K^2 + 2(1-\ubar{\alpha})^2 \| \theta \|_K \| \tilde{\theta} \|_K - (1-\ubar{\alpha})^2 \| \theta \|_K^2 \leq 0
\end{align*}
for all $\tilde{\theta} \in \tilde{\Theta}^{\text{\sc bls}}$. Since this quadratic equation is convex $\ubar{\alpha}(2-\ubar{\alpha})>0$, it is positive between its roots, which can be obtain from the quadratic formula
\begin{align*}
-\frac{1-\ubar\alpha}{\ubar{\alpha}} \| \theta \|_K \leq \| \tilde{\theta} \|_K \leq \frac{1-\ubar\alpha}{2-\ubar{\alpha}} \| \theta \|_K
\end{align*}
for all $\tilde{\theta} \in \tilde{\Theta}^{\text{\sc bls}}$. Since $\ubar{\alpha} \in (0,1)$ and $\| \tilde{\theta} \|_K \geq 0$, the lower-bound is redundant. Thus, $\tilde{\Theta}^{\text{\sc bls}} \subseteq \tilde{\Theta}^{\max}$ if $\| \tilde{\theta} \|_K \leq \frac{1-\ubar\alpha}{2-\ubar{\alpha}} \| \theta \|_K$ for all $\tilde{\theta} \in \tilde{\Theta}^{\text{\sc bls}}$. For $\tilde{\theta} \in \tilde{\Theta}^{\text{\sc bls}}$, the norm $\| \tilde{\theta} \|_K$ satisfies
\begin{align*}
\| \tilde{\theta} \|_K = \| K^{\frac{1}{2}} \tilde{\theta} \| 
&\leq \| K^{\frac{1}{2}} \| \| \tilde{\theta} \| =  \| K^{\frac{1}{2}} \| \| \Lambda \Lambda^{-1} \tilde{\theta} \| \\
&\leq  \| K^{\frac{1}{2}} \| \| \Lambda \| \| \Lambda^{-1} \tilde{\theta} \| \leq  \| K^{\frac{1}{2}} \| \| \Lambda \|
\end{align*}
where $\| \Lambda^{-1} \tilde{\theta} \|  \leq 1$ by the definition~\eqref{eq:estimation-error} of the set $\tilde{\Theta}^{\text{\sc bls}}$. Thus, $$\tilde{\Theta}^{\text{\sc bls}} \subseteq \tilde{\Theta}^{\max} \;\; \text{ if } \;\;  \| K^{\frac{1}{2}} \| \| \Lambda \| \leq \frac{1-\ubar\alpha}{2-\ubar{\alpha}} \| \theta \|_K$$. Or equivalently, $\tilde{\Theta}^{\text{\sc bls}} \subseteq \tilde{\Theta}^{\max}$ when the information matrix~\eqref{eq:estimator-information} is sufficiently large 
\begin{align}
\label{pf:information-bound}
\Lambda^{-1} \succeq \frac{ 2 - \ubar{\alpha} }{ 1-\ubar{\alpha} } \frac{  \| K^{\frac{1}{2}} \| }{ \| \theta \|_K } I
\end{align}
where $\| \theta \|_K = \| \nabla \cost \|_K \neq 0$ by the hypothesis of this lemma. In other words, the \ac{esc} controller leaves the exploration mode when there is enough~\eqref{pf:information-bound} information to reliably estimate the gradient. Next, we prove that the information matrix~\eqref{eq:estimator-information} is sufficiently large~\eqref{pf:information-bound} after a finite period of time $T < \infty$ in the exploration mode and non-zero dither amplitude $\delta >0$. Consider the following two conditions:
\begin{enumerate}
\item The weightings~\eqref{eq:weighting} are sufficiently large
\begin{align}
\label{pf:weighting-bounds}
w_k \geq \frac{ 2 - \ubar{\alpha} }{ 1-\ubar{\alpha} } \frac{  \| K^{\frac{1}{2}} \| }{ \| \theta \|_K }\frac{1}{\rho \| \Delta y_k \|^2} 
\end{align}
\item The output transients $\Delta  y_k$ are \ac{pe}
\begin{align}
\label{pf:output-pe}
\frac{1}{N} \sum_{k=t-N}^t \frac{\Delta y_k \Delta y_k^\trans}{\| \Delta y_k \|^2} \succeq \rho I
\end{align}
\end{enumerate}
where $\rho > 0$ quantifies the excitement of the outputs $\Delta  y_k$. If conditions~\eqref{pf:weighting-bounds} and~\eqref{pf:output-pe} hold then the information matrix~\eqref{eq:estimator-information} is sufficiently large~\eqref{pf:information-bound} to allow the controller~\eqref{eq:controller} to enter the exploitation mode since 
\begin{align*}
\Lambda_t^{-1} 
&= \frac{1}{N} \sum_{k=t-1}^{t-N} w_k \Delta y_k \Delta y_k^\trans \\
&\succeq \frac{1}{N} \sum_{k=t-1}^{t-N} \frac{ 2 - \ubar{\alpha} }{ 1-\ubar{\alpha} } \frac{  \| K^{\frac{1}{2}} \| }{ \| \theta \|_K } \frac{\Delta y_k \Delta y_k^\trans}{\rho \| \Delta y_k\|^2} 
\succeq \frac{ 2 - \ubar{\alpha} }{ 1-\ubar{\alpha} } \frac{  \| K^{\frac{1}{2}} \| }{ \| \theta \|_K } I. 
\end{align*} 

Thus, we next prove condition~\eqref{pf:weighting-bounds} holds after a finite-time $T_0 < \infty$ and non-zero dither amplitude $\delta >0$. According to Corollary~\ref{cor:exploration}, the equilibrium state $\bar{x} = \pi(\bar{r})$ is \ac{iss} where the reference is constant $r_t = \bar{r}$ when the controller~\eqref{eq:controller} is in the exploration mode. By the definition of \ac{iss}, we have 
\begin{align*}
\| x_t - \bar{x} \| \leq \beta( \| x_0 - \bar{x} \|, t) + \gamma(\sup_{t} \| d_t \|)
\end{align*}
where $\beta \in \KL$ and $\gamma\in\K$. Thus, for any initial condition, there exists a finite-time $T_0 < \infty$ and non-zero dither amplitude $\delta >0$ such that 
\begin{align*}
\| x_t - \bar{x} \| 
\leq \beta( \| x_0 - \bar{x} \|, T_0) + \gamma( \delta ) 
\leq   \frac{\rho}{\ell_g \| \tilde{H} \|}  \frac{ 2 - \ubar{\alpha} }{ 1-\ubar{\alpha} } \frac{  \| K^{\frac{1}{2}} \| }{ \| \theta \|_K }.
\end{align*}
for all $t \geq T_0 \in \naturals$ and $\| d_t \| \leq \delta$ where $\ell_g$ is a Lipschitz bound on $g$. Since the plant~\eqref{eq:plant} output map $g$ is Lipschitz continuous, we have the following bound on the tracking error $e_t$
\begin{subequations}
\label{pf:iss-bounds}
\begin{align}
\| e_t \| 
&= \| y_t - r_t \| = \| g(x_t) - g(\bar{x}) \| \leq \ell_g \| x_t - \bar{x} \| \\
&\leq  \frac{\rho}{\| \tilde{H} \|} \frac{ 2 - \ubar{\alpha} }{ 1-\ubar{\alpha} } \frac{  \| K^{\frac{1}{2}} \| }{ \| \theta \|_K }.
\end{align}
Likewise, we can obtain a bound on the output transients $y_k - y_{k-1}$ 
\begin{align}
\| \Delta y_k \| &= \| g(x_k) - g(\bar{x}) + g(\bar{x}) - g(x_{k-1}) \| \\
& \leq \frac{2\rho}{\| \tilde{H} \|} \frac{ 2 - \ubar{\alpha} }{ 1-\ubar{\alpha} } \frac{  \| K^{\frac{1}{2}} \| }{ \| \theta \|_K }.
\notag
\end{align}
\end{subequations}
Substituting the bounds~\eqref{pf:iss-bounds} into the weightings~\eqref{eq:weighting} produces
\begin{align*}
w_k^{-1} 
&= \frac{1}{2} \| \Delta y_k \| \| \Delta y_k \|_{\tilde{H}}  \Big( \| e_t \|_{\tilde{H}} + \frac{1}{2} \| \Delta y_k \|_{\tilde{H}} \Big)  \\
&= \frac{ 2 - \ubar{\alpha} }{ 1-\ubar{\alpha} } \frac{  \| K^{\frac{1}{2}} \| }{ \| \theta \|_K } \rho \| \Delta y_k \|^2
\end{align*}
Thus, condition~\eqref{pf:weighting-bounds} holds after finite-time $T_0 < \infty$ and for a non-zero dither amplitude $\delta >0$. 

Next, we prove condition~\eqref{pf:output-pe} holds for a \ac{pe} dither $d_t$. Since the plant~\eqref{eq:plant} is controllable, a \ac{pe} dither $d_t$ will produce a \ac{pe} state $x_t$ after a finite-period $n_c$ where $n_c$ is the controllability index. Likewise, since the plant~\eqref{eq:plant} is observable, the \ac{pe} state will produce a \ac{pe} output sequence $\Delta y_k$ after a finite-period $n_o$ where $n_o$ is the observability index. Thus, condition~\eqref{pf:output-pe} holds for some $\rho$ after a finite period $n_c + n_o$. Note that the \ac{pe} parameter $\rho$ does not depend on the dither amplitude $\delta$ since the data $\{ \Delta y_k \}_{k=t+n_c + n_o}^{t+n_c+n_o+N}$ is normalized in~\eqref{pf:output-pe}. 

Thus, we have proven that conditions~\eqref{pf:weighting-bounds} and~\eqref{pf:output-pe} hold after a finite period $T = \max\{T_0,n_c + n_o\}$ for a \ac{pe} and sufficiently small $\| d_t \| \leq \delta$ dither $d_t$. This means that the information matrix~\eqref{eq:estimator-information} is sufficiently large~\eqref{pf:information-bound}. Therefore, $\tilde{\Theta}^{\text{\sc bls}} \subseteq \tilde{\Theta}^{\max}$ allowing the controller~\eqref{eq:controller} to reenter the exploitation mode. 
\end{proof}

Lemma~\ref{lemma:time-bound} shows that \ac{pe} data $\{ \cost_k, y_k \}_{k=t-N}^t$ is necessary, but not sufficient to accurately estimate the cost gradient $\nabla\cost(r_t)$. For an accurate estimates, the data must also be sufficiently local $y_t \approx r_t$ and sufficiently close to equilibrium $y_t \approx y_{t-1}$. For instance, data $\{ \cost_k, y_k \}_{k=t-N}^t$ collected far $\| y_k - r_t \| \gg 0$ from the set-point cannot be used to accurately estimate the gradient $\nabla \cost(r_t)$ at $r_t$ since the cost is nonlinear $\ubar{H} \preceq \nabla^2 \cost \preceq \bar{H}$. Other \ac{esc} controllers (e.g.~\cite{Guay2015, Ariyur2003}) address this issue by reducing the bandwidth of the \ac{esc} controller. In contrast, our \ac{esc} controller only reduces the bandwidth when it detects $\alpha_t^\star = 0$ that the gradient cannot be reliable estimated. Thus, we say our \ac{esc} controller has an adaptive separation of time-scales since the adaptive step-size~\eqref{eq:step-size} throttles the \ac{esc} controller~\eqref{eq:controller} to allow the plant~\eqref{eq:plant} settle providing better data for the gradient estimator~\eqref{eq:estimator}.

The final result necessary for the proof of Theorem~\ref{thm:main} connects the game-theoretic step-size~\eqref{eq:step-size-generic} used in Corollary~\ref{cor:exploration} with the step-size~\eqref{eq:step-size} used by the \ac{esc} controller~\eqref{eq:controller}. 

\begin{corollary}
\label{cor:step-size}
Let Assumptions~\ref{assume:plant} and~\ref{assume:cost} hold. For the estimation error set~\eqref{eq:estimation-error}, the optimal step-size~\eqref{eq:step-size-generic} has the closed-form~\eqref{eq:step-size}.
\end{corollary}

\begin{proof}
For the estimation error set~\eqref{eq:estimation-error}, the optimization problem~\eqref{eq:step-size-generic} used to select the adaptive step-size $\alpha_t^\star$ can be reformulated as
\begin{align}
\label{eq:step-size-optimization}
\max_{\tilde{\theta} \in \tilde{\Theta}_t^{\text{\sc bls}}} ~ \tilde{\theta}^\trans K \hat{\theta} 
= 
\max_{\| z \|^2 \leq 1}   z^\trans \Lambda_t^{\frac{1}{2}} K \hat{\theta}_t 
\end{align}
where $z = \Lambda_t^{-1/2}\tilde{\theta}$ is a change-of-variables. The optimization problem~\eqref{eq:step-size-optimization} has a closed-form solution, namely $$z^\star = \pm \Lambda_t^{\frac{1}{2}} K \hat{\theta}_t / \| \Lambda_t^{\frac{1}{2}} K \hat{\theta}_t \|$$ i.e. the unit vector $\| z^\star \|\leq 1$ aligned with the cost $\Lambda_t^{\frac{1}{2}} K\hat{\theta}_t$. Substituting $z^\star$ into~\eqref{eq:step-size-generic} yields~\eqref{eq:step-size}. 
\end{proof}

Corollary~\ref{cor:step-size} shows that the step-size~\eqref{eq:step-size} used by the \ac{esc} controller~\eqref{eq:controller} is closed-form solution of the game-theoretic optimal step-size~\eqref{eq:step-size-generic} for the particular bounds~\eqref{eq:estimation-error} on the gradient estimation errors of the \ac{bls} estimator~\eqref{eq:estimator}.

Finally, we can prove Theorem~\ref{thm:main}. 

\begin{proof}[Proof of Theorem~\ref{thm:main}]
Since the \ac{esc} controller~\eqref{eq:controller} switches between the exploration and exploitation modes, we will use switched systems theory to prove stability. 

Let $t_i$ for $i \in \naturals$ denote the time-indices where the \ac{esc} controller~\eqref{eq:controller} switches modes. Without loss of generality, assume that the system is in the exploration mode at even time-indices $t_{2k}$ for $k \in \naturals$. From Corollaries~\ref{cor:exploitation} and~\ref{cor:exploration}, the common Lyapunov function~\eqref{eq:lyapunov} satisfies
\begin{align}
\label{eq:iss-switching}
V(\tilde{x}_{2k+2},\tilde{r}_{2k+2}) &-  V(\tilde{x}_{2k},\tilde{r}_{2k}) \\
& \leq - q(\| \tilde{x}_{2k}, \tilde{r}_{2k} \|) + \sigma_s(\|d_{2k},\dots,d_{2k+1}\|) \notag
\end{align}
where $\tilde{x}_{2k} = \tilde{x}_{t_{2k}}$ and $\tilde{r}_{2k} = \tilde{r}_{t_{2k}}$ and $q$ is obtained by summing $q_x$ and $q_r$ along the closed-loop trajectories
\begin{align*}
q(\| \tilde{x}_{2k}, \tilde{r}_{2k} \|) \!=\! \sum_{t=t_{2k}}^{t_{2k+1}} q_x^1( \| \tilde{x}_t\| ) \!+\! \sum_{t=t_{2k+1}}^{t_{2k+2}} q_x^2( \| \tilde{x}_t\| ) \!+\! q_r^2( \| \tilde{r}_t\| ).
\end{align*}
where $q$ is a class-$\mathcal{K}_\infty$ function of the states $( \tilde{x}_{2k},\tilde{r}_{2k})$ at the $2k$ switching instance $t_{2k}$. Likewise, the summation $$\sigma_s = \sum_{t=t_{2k}}^{t_{2k+2}} \sigma(\| d_t\|)$$ is a class-$\K$ function of the dither $d_t$ between times $t_{2k}$ and $t_{2k+2}$. Thus, by~\eqref{eq:iss-switching} and Proposition~2.3 in~\cite{Jiang2001}, there exists $\beta \in \K_\infty$ and $\gamma \in \K$ such that 
\begin{align*}
\big\| \tilde{x}_{2k}, \tilde{r}_{2k} \big\| 
\leq 
\beta\left( \big\|  \tilde{x}_{0}, \tilde{r}_{0} \big\| | , 2k \right) +  \gamma( \delta )
\end{align*}
Therefore, the state $(x,r)$ of the closed-loop system~\eqref{eq:plant} and~\eqref{eq:controller}-\eqref{eq:step-size} converges to a neighborhood of the optimal equilibrium~\eqref{eq:equilibrium} as the \textit{switching index} $k$ goes to infinity $k \rightarrow \infty$. Thus, we need to prove that switching index goes to infinity $k \rightarrow \infty$ as time goes to infinity $t \rightarrow \infty$ i.e. we do not become trapped in the exploration mode.

According to Lemma~\ref{lemma:time-bound},  if $k \not\rightarrow \infty$ then $t \not\rightarrow \infty$ since $t \leq \sum_{k=0}^{\sup k<\infty} T_k < \infty$ where $T_k < \infty$. Thus, by the contrapositive $k \rightarrow \infty$ as $t \rightarrow \infty$. Therefore, the optimal equilibrium~\eqref{eq:equilibrium} of the closed-loop system~\eqref{eq:plant} and~\eqref{eq:controller}-\eqref{eq:step-size} is \ac{iss}. 
\end{proof}

\section{Numerical Examples}
\label{sec:examples}

In this section, we demonstrate our \ac{esc} controller through a series of numerical examples. 

\subsection{Illustrative Example}

In this section, we demonstrate our \ac{esc} controller for a simple linear system with an unknown quadratic cost function. The purpose of this example is to illustrate our \ac{esc} controller~\eqref{eq:controller}-\eqref{eq:step-size} using classical control theory. 

The plant~\eqref{eq:plant} is an under-damped second-order linear system
\begin{align}
\label{eq:illustrative-plant}
\ddot{y} + 2 \zeta \omega_n \dot{y} + \omega_n^2 y = \omega_n^2 u
\end{align}
where $\zeta = 0.1$ and $\omega_n = 1.0$. The cost $\cost$ is a quadratic
\begin{align}
\label{eq:illustrative-cost}
\cost(y) = \tfrac{1}{2} H (y-y^\star)^2
\end{align}
where $y^\star$ is the optimal and $H \in \reals$ is the Hessian. 

\begin{figure}[ht]
\begin{center}
\includegraphics[width = 0.8\columnwidth]{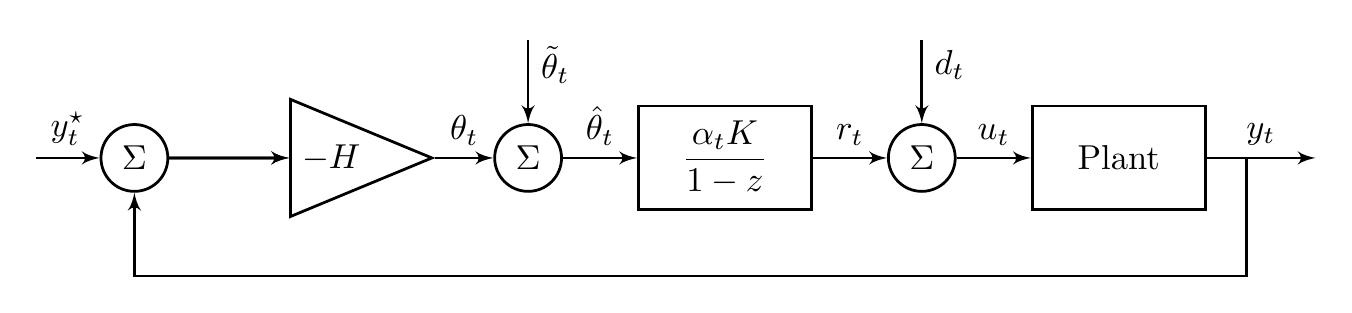}
\end{center}
\caption{The feedback-loop created by descending the gradient $H(y-y^\star)$ of the quadratic cost function~\eqref{eq:illustrative-cost}.}
\label{fig:illustrative-feedback}
\end{figure}

First, we consider the ideal case where the gradient $\nabla \cost(y) = H (y-y^\star)$ has been perfectly estimated $\tilde{\theta}_t=0$. Since the gradient $\nabla \cost(y) = H (y-y^\star)$ of the quadratic cost~\eqref{eq:illustrative-cost} is a linear function of the output $y$, we obtain the feedback-loop shown in Figure~\ref{fig:illustrative-feedback}. Even though the gradient is perfectly estimated $\tilde{\theta}_t=0$ and the plant~\eqref{eq:illustrative-plant} is open-loop stable, the integral-action of the \ac{esc} controller can destabilize the closed-loop system, as shown by the root-locus in Figure~\ref{fig:root-locus-1}. In particular, the Newton-step controller gain $\alpha_t K = H^{-1}$, which provides one-step convergence to the optimal for static optimization, destabilizes this dynamic optimization, as shown by root-locus in Figure~\ref{fig:root-locus-1}. In contrast, the proposed controller gain~\eqref{eq:esc-gain} provides closed-loop stability for perfect gradient estimates. 

\begin{figure}[!ht]
\centering\subfloat[Cost]{
\includegraphics[width = 0.49\columnwidth]{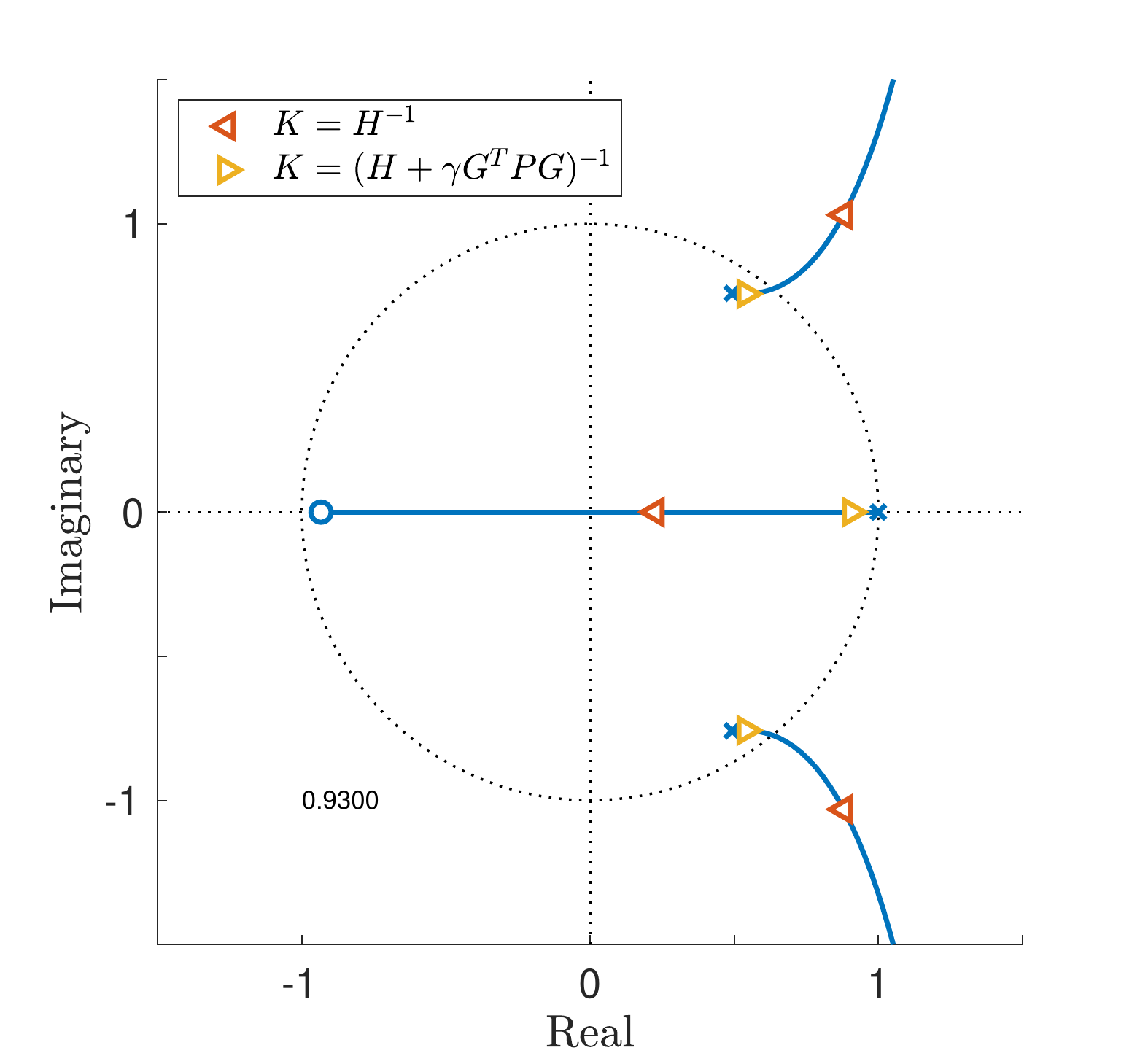}
\label{fig:root-locus-1}
} \subfloat[Step-Size]{
\includegraphics[width = 0.49\columnwidth]{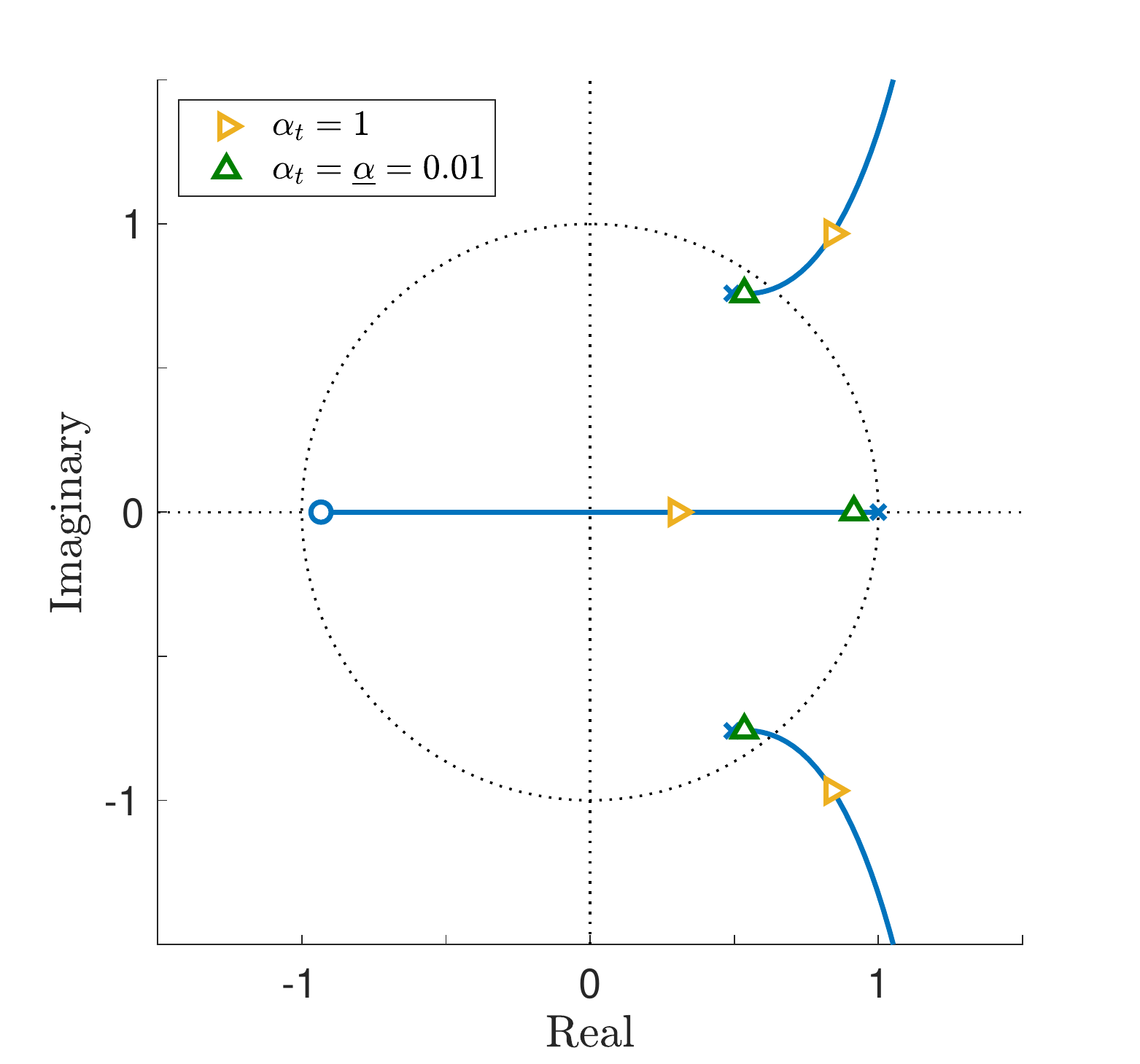}
\label{fig:root-locus-2}
}
\caption{Root-locus plots demonstrating that the integral-action of the \ac{esc} controller~\eqref{eq:controller} can destabilize the linear plant~\eqref{eq:illustrative-plant}. Note that these are discrete-time root-locus plots.}
\label{fig:root-locus}
\end{figure}

Unfortunately, our instability issues re-emerge when we consider imperfect gradient estimates $\hat{\theta}_t \neq \nabla \cost_t$. In particular, the worst-case gradient estimation error $$\tilde{\theta}_t^\star = \tfrac{1-\ubar{\alpha}}{\ubar{\alpha}} \nabla \cost_t \in \tilde{\Theta}_t = \tilde{\Theta}_t^{\max}$$ amplifies the feedback caused by the gradient $$\hat{\theta}_t = \nabla \cost_t + \tilde{\theta}_t = \frac{1}{\ubar{\alpha}} \nabla \cost_t = \frac{H}{\ubar{\alpha}}(y-y^\star)$$ where $\ubar{\alpha} \ll 1$. This increases the loop-gain, leading to instability, as shown by the root-locus in Figure~\ref{fig:root-locus-2}. Fortunately, our adaptive step-size~\eqref{eq:step-size} will compensate $\alpha_t^\star = \ubar{\alpha}$ for the expansion $\tfrac{1}{\ubar{\alpha}}$ of the loop-gain to restore stability, as shown by the root-locus in Figure~\ref{fig:root-locus-2}.

\begin{figure}[ht]
\centering
\subfloat[Cost]{
\includegraphics[width = 0.49\columnwidth]{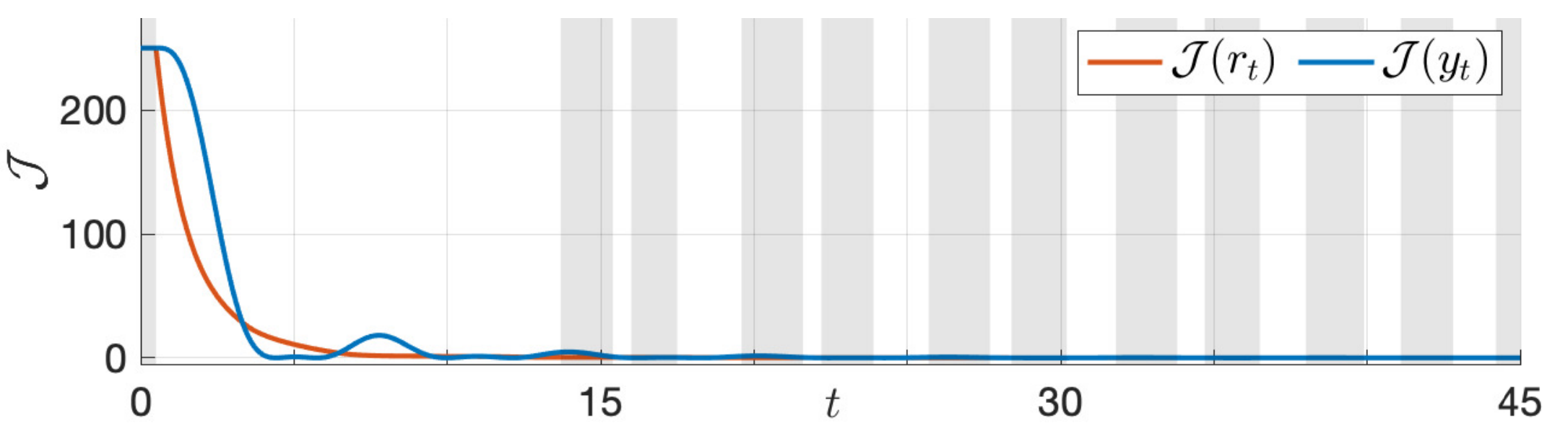}
\label{fig:illustrative-cost}
}
\subfloat[Step-Size]{
\includegraphics[width = 0.49\columnwidth]{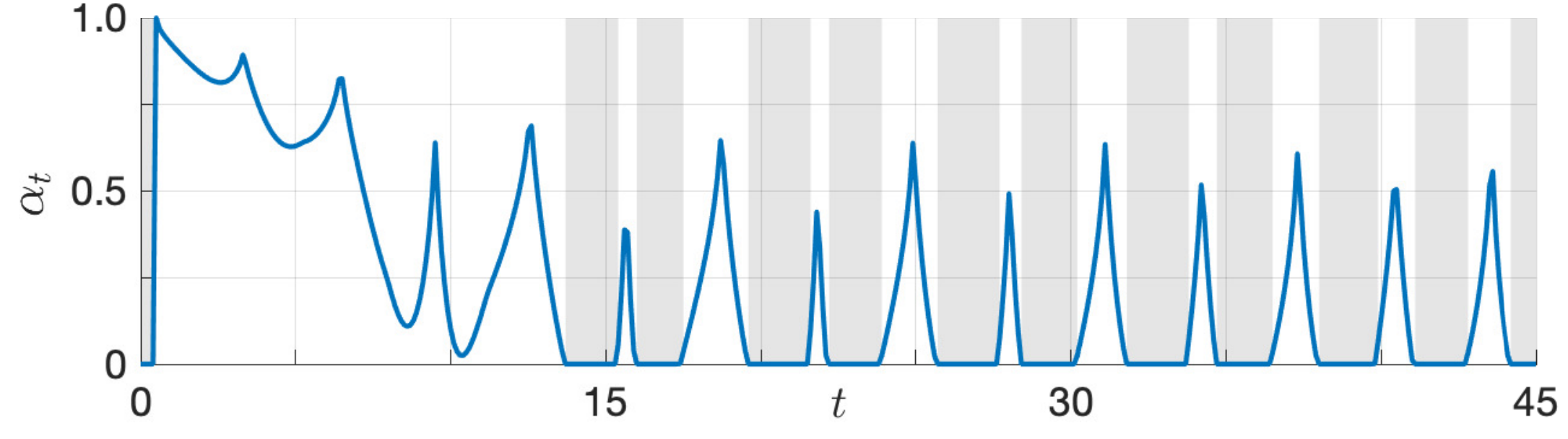}
\label{fig:illustrative-step}
} \\
\subfloat[Output]{
\includegraphics[width = 0.49\columnwidth]{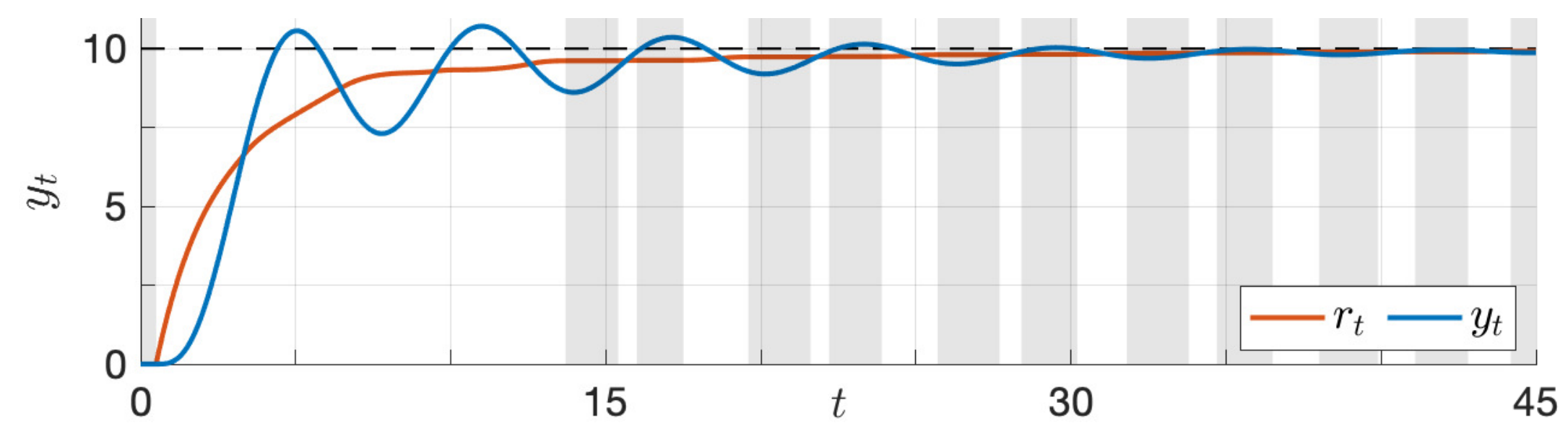}
\label{fig:illustrative-output}
}
\subfloat[Error]{
\includegraphics[width = 0.49\columnwidth]{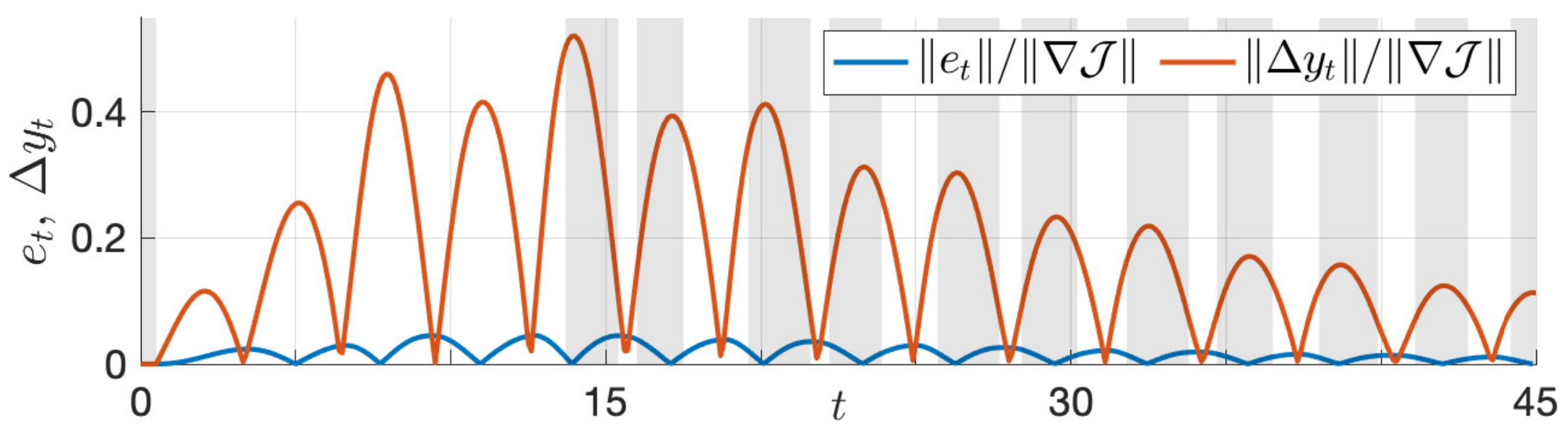}
\label{fig:illustrative-errors}
} 
\caption{Closed-loop ~\eqref{eq:illustrative-plant} and~\eqref{eq:controller}-\eqref{eq:step-size} simulation results. Shaded regions indicate when the \ac{esc} controller is in the exploration mode. }
\label{fig:illustrative}
\end{figure}

Finally, we demonstrate our \ac{esc} controller for the plant~\eqref{eq:illustrative-plant} and cost~\eqref{eq:illustrative-cost}. The \ac{bls} estimator~\eqref{eq:estimator} has a batch horizon $N=5$ and bounds $\ubar{H} = 0$ and $\bar{H} = 10$ on the actual gradient $H = 5$. A dither $d_t = 0.001\sin(t)$ was used to provide persistency of excitation. 

Simulations for the closed-loop system~\eqref{eq:illustrative-plant} and~\eqref{eq:controller}-\eqref{eq:step-size} are shown in Figure~\ref{fig:illustrative}. As shown in Figure~\ref{fig:illustrative-output}, the plant output converges $y_t \rightarrow y^\star$ to the optimal $y^\star = 10$. Since the plant~\eqref{eq:illustrative-plant} is under-damped, the output $y_t$ oscillates and the measured cost $\cost(y_t)$ converges non-monotonically to the optimal value. However, the set-point cost $\cost(r_t)$ is monotonically decreasing as shown in Figure~\ref{fig:illustrative-cost}. Figure~\ref{fig:illustrative-step} shows the step-size~\eqref{eq:step-size}. Initially, the step-size~\eqref{eq:step-size} is large since the gradient $\nabla \cost = H(y-y^\star)$ is large far from the optimal $| y - y^\star | \gg 0$. Thus, an accurate gradient estimate is not required to confidently descend. When the gradient becomes small, the controller often enters the exploration mode, indicated by the shaded regions in Figure~\ref{fig:illustrative}. As shown in Figure~\ref{fig:illustrative-errors}, the periods when the \ac{esc} controller is in the exploration mode $\alpha_t =0$ correspond to periods when the plant is far from equilibrium $\| \Delta y_t \| / \| \nabla \cost \| \gg 0$. 

\subsection{Practical Example: Drone Leak Inspection}

In this section, we apply our \ac{esc} controller~\eqref{eq:controller}-\eqref{eq:step-size} to the problem of an autonomous drone searching for the source of an airborne pollutant leak. 

The plant dynamics~\eqref{eq:plant} model the closed-loop planar motion of a quadrotor drone. We use a standard model of the quadrotor dynamics e.g.~\cite{Bresciani2008}. For simplicity, we only consider the movement of the drone in the plane i.e. the vertical position and orientation dynamics are ignored. The quadrotor is equip with {\sc gps} that measures its planar location $y \in \reals^2$ and an integrated controller that moves the drone to a commanded location $y_t \rightarrow r \in \reals^2$. Thus, the plant satisfies Assumption~\ref{assume:plant}. 

The objective of the \ac{esc} controller~\eqref{eq:controller}-\eqref{eq:step-size} is to move the drone to the source of a pollutant leak. The cost function $\cost(y)$ optimized by the \ac{esc} controller~\eqref{eq:controller} is the location $y$ dependent measured concentration of pollutant in the air. Since our \ac{esc} controller minimize the cost function, we will consider the negative pollutant concentration. The negative pollutant concentration is modeled using a Gaussian plume model~\cite{Beychok1994}
\begin{subequations}
\label{eq:drone-cost}
\begin{align}
\cost(y) =- \frac{1}{\sqrt{2 \pi \sigma}} \exp\Big( - \tfrac{1}{2} (y-y^\star)^\trans \Sigma^\dagger(y-y^\star)  \Big) 
\end{align}
where $y^\star = [200,100]^\trans$ meters is the planar location of the leak and $\Sigma^\dagger(y-y^\star)$ is the pseudo-inverse of the covariance of pollutant concentration~\cite{Beychok1994}
\begin{align}
\Sigma(y-y^\star) = 
\begin{cases}
\sigma^2 (I - d d^\trans ) & \text{ if } d^\trans (y - y^\star) \geq 0 \\
\sigma_0^2 I & \text{ otherwise} 
\end{cases}
\end{align}
\end{subequations}
where $v = 10$ meters/second is the wind velocity (about $20$ knots) and $d = [\cos(-\pi/4),\sin(-\pi/4)]^\trans$ is the wind direction. The cost~\eqref{eq:drone-cost} says that the pollutant concentration has a Gaussian distribution in the cross-wind direction $(I-dd^\trans)(y - y^\star)$. Note that the matrix $(I-dd^\trans) = (I-dd^\trans)^2$ is idempotent. The covariance $\sigma = \sigma_0 + d^\trans (y-y^\star)/2$ of this Gaussian grows linearly with the distance $d^\trans (y-y^\star)$ along the wind-direction $d$ from the source $y^\star$. In the anti-wind direction $d^\trans (y - y^\star) < 0$, the covariance is constant $\Sigma = \sigma_0^2 I$. Note that although the cost~\eqref{eq:drone-cost} is not convex, it locally satisfy Assumption~\ref{assume:cost}. 

Our discrete-time \ac{esc} controller~\eqref{eq:controller}-\eqref{eq:step-size} is executed at a rate of $20$ Hertz. The \ac{bls} estimator~\eqref{eq:estimator} estimates the gradient $\nabla \cost$ from the past $1$ second of data, thus $N=20$. The estimator uses the bounds $\ubar{H} = -0.3 I$ and $\bar{H} = 0.15I$ on the curvature $\nabla^2 \cost$, which is approximately $5 \times$ the actual curvature bounds. The controller gain~\eqref{eq:esc-gain} is $K = I$. The dither $d_t \sim \mathcal{N}(0,1)$ is a normal distributed random variable with covariance of $1$ meter.

\begin{figure}[!ht]
\centering
\includegraphics[width = 0.4\columnwidth]{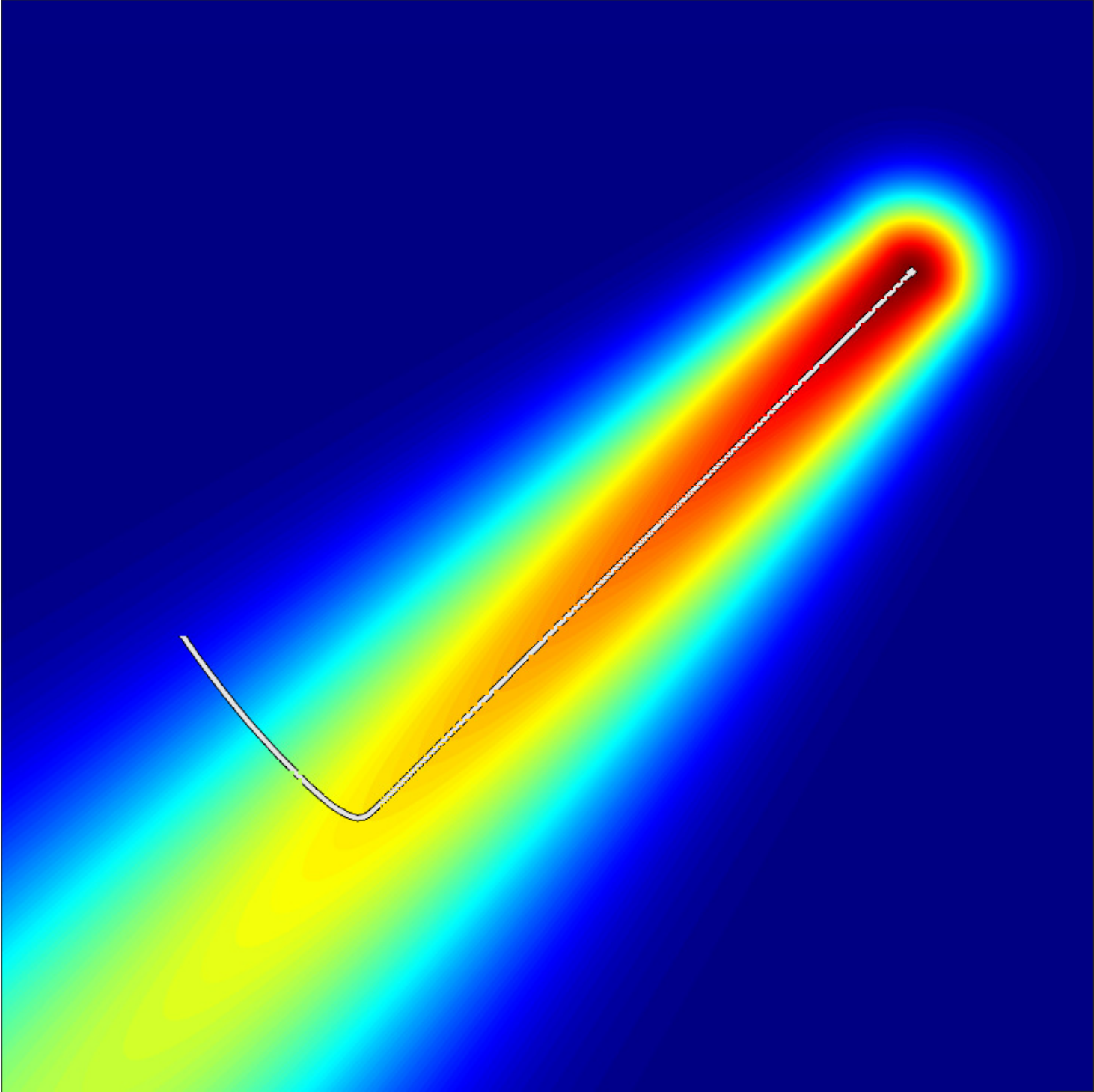}
\caption{Pollutant concentration and drone path. Red indicates high concentration while blue indicates low concentration. Pollutant is blown by the wind to form a plume. The \ac{esc} controller drives the drone into the plume and follows the plume of the source of the leak.}
\label{fig:drone-trajectory}
\end{figure}

Closed-loop simulation results are shown in Figures~\ref{fig:drone-trajectory} and~\ref{fig:drone-outputs}. Figure~\ref{fig:drone-trajectory} shows the pollutant concentration and the path of the drone. The drone starts outside of the pollutant plume and moves perpendicular to the wind-direction into the plume stream.  Once the drone enters the pollutant stream, is proceed against the wind direction to the source of the pollutant.

\begin{figure}[!ht]
\centering
\subfloat[Cost]{
\includegraphics[width = 0.48\columnwidth]{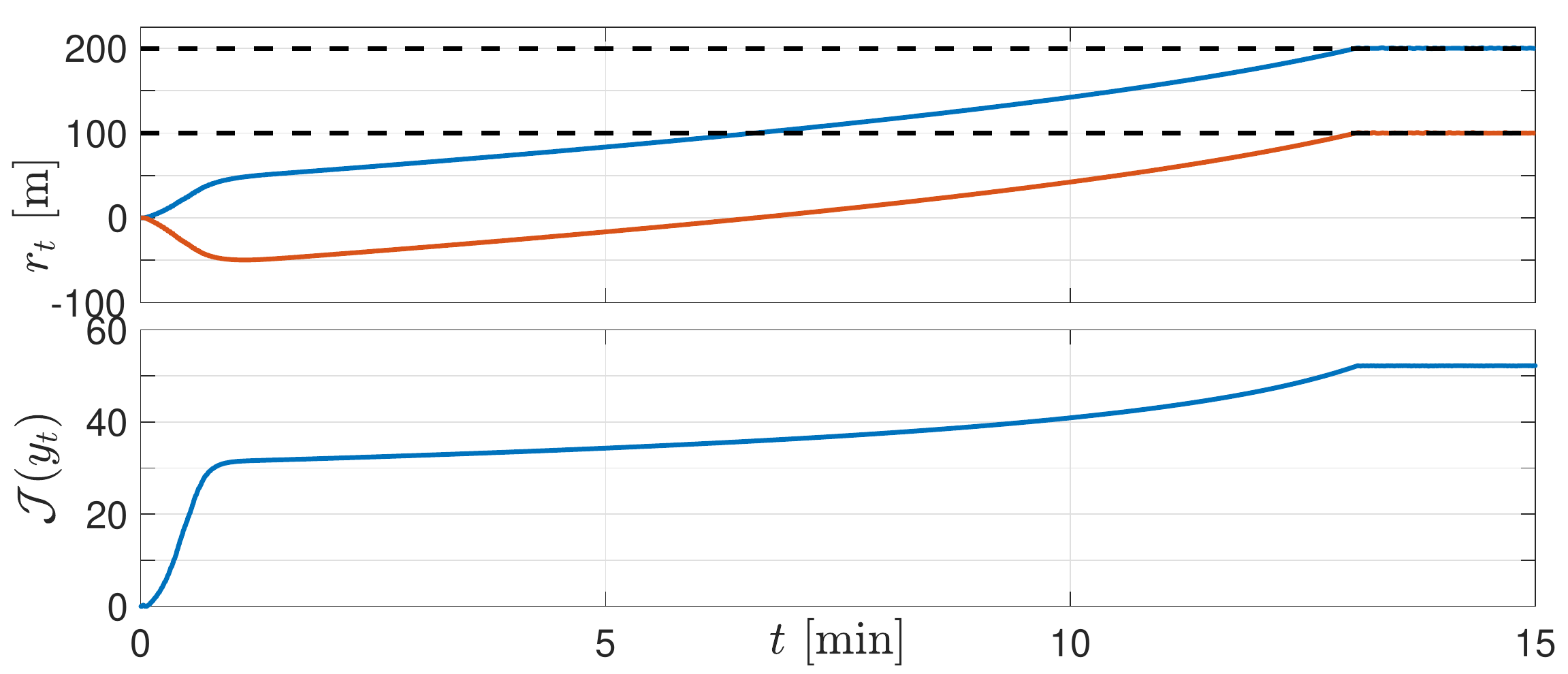}
\label{fig:drone-cost}
}
\subfloat[Step-Size]{
\includegraphics[width = 0.48\columnwidth]{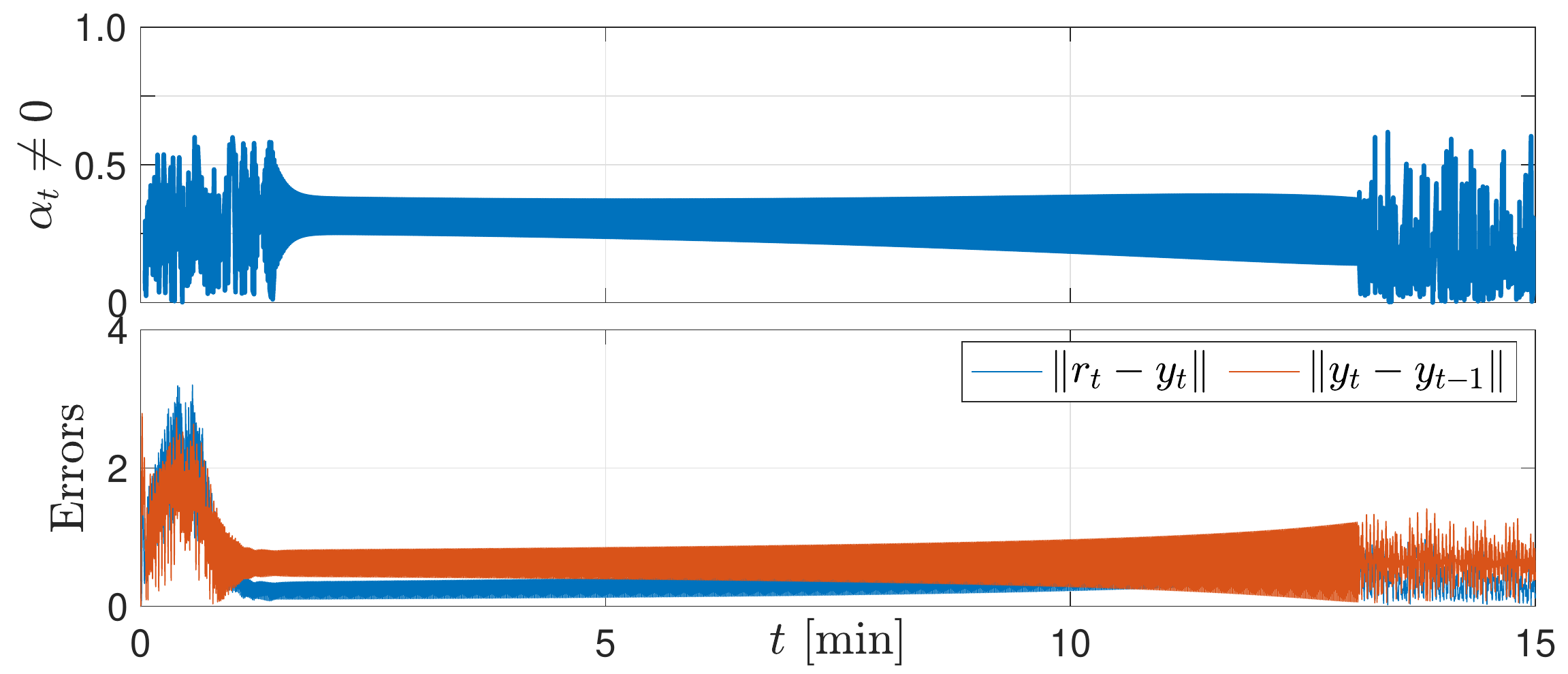}
\label{fig:drone-step}
}
\caption{(a) Reference location $r_t$ for the drone and the resulting pollutant concentration $\cost(y_t)$ over the $15$ minute simulation. (b) Step-size $\alpha_t$ used by the \ac{esc} controller and the norms of the tracking error $e_t = r_t - y_t$ and drone transients $\Delta y_t = y_t - y_{t-1}$. }
\label{fig:drone-outputs}
\end{figure}

Figure~\ref{fig:drone-cost} shows the location set-point $r_t$ and concentration $\cost(y_t)$ measured at $y_t$ as a function of time $t$. The drone converges $r_t \rightarrow r^\star$ to the location $r^\star = y^\star$ of the pollutant leak as shown by the dashed black lines in Figure~\ref{fig:drone-cost}. Likewise, the measured pollutant concentration $\cost(y_t)$ converges to the maximum. The step-size~\eqref{eq:step-size} is shown in Figure~\ref{fig:drone-step}. Since the step-size~\eqref{eq:step-size} is zero $\alpha_t = 0$ approximately $90\%$ of the time, we only plot it for the time-instance $t$ when it is non-zero $\alpha_t \neq 0$. Since the \ac{esc} controller runs at $20$ Hertz, the step-size is non-zero $\alpha_t \neq 0$ on average $12$ times per minute, meaning that the estimated location $r_t$ of the leak source is persistently and frequently updated.

\subsection{Benchmark Examples}

In this section, we compare our \ac{esc} controller with existing methods using three benchmark examples from the literature. 

\subsubsection{1-D Benchmark}

In this section, we demonstrate our \ac{esc} controller for the $1$ state benchmark example from~\cite{Hunnekens2014}. The plant dynamics are
\begin{subequations}
\label{eq:benchmark1-plant}
\begin{align}
\dot{x} &= - x + u \\
y &= x
\end{align}
\end{subequations}
The plant~\eqref{eq:benchmark1-plant} is a stable linear system and therefore satisfies Assumption~\ref{assume:plant}. 
The unknown cost function is
\begin{align}
\label{eq:benchmark1-cost}
\cost(y) = 3 - \frac{1}{\sqrt{1+(y-2)^2}}.
\end{align}
Although the cost~\eqref{eq:benchmark1-cost} is non-convex, it locally satisfies Assumption~\ref{assume:cost}. Our \ac{esc} controller~\eqref{eq:controller} used the gain $K=0.5$ and a sample rate of $10$ Hertz. The \ac{bls} estimator~\eqref{eq:estimator} had an estimation horizon of $N=5$ and bounds $\ubar{H} = -2$ and $\bar{H} = 2$ on the curvature $\nabla^2 \cost$ of the cost. The dither $d_t = 0.001 \sin(t)$ was used to provide persistency of excitation. 

\begin{figure}[!ht]
\centering
\subfloat[Cost]{
\includegraphics[width = 0.48\columnwidth]{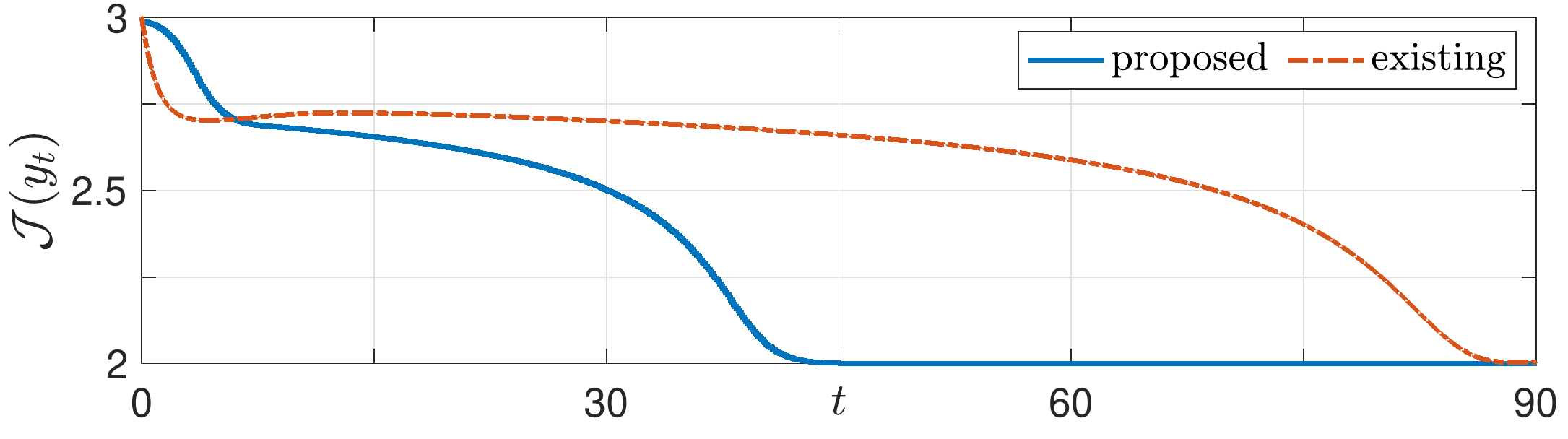}
\label{fig:benchmark1-cost}
}
\subfloat[Step-Size]{
\includegraphics[width = 0.48\columnwidth]{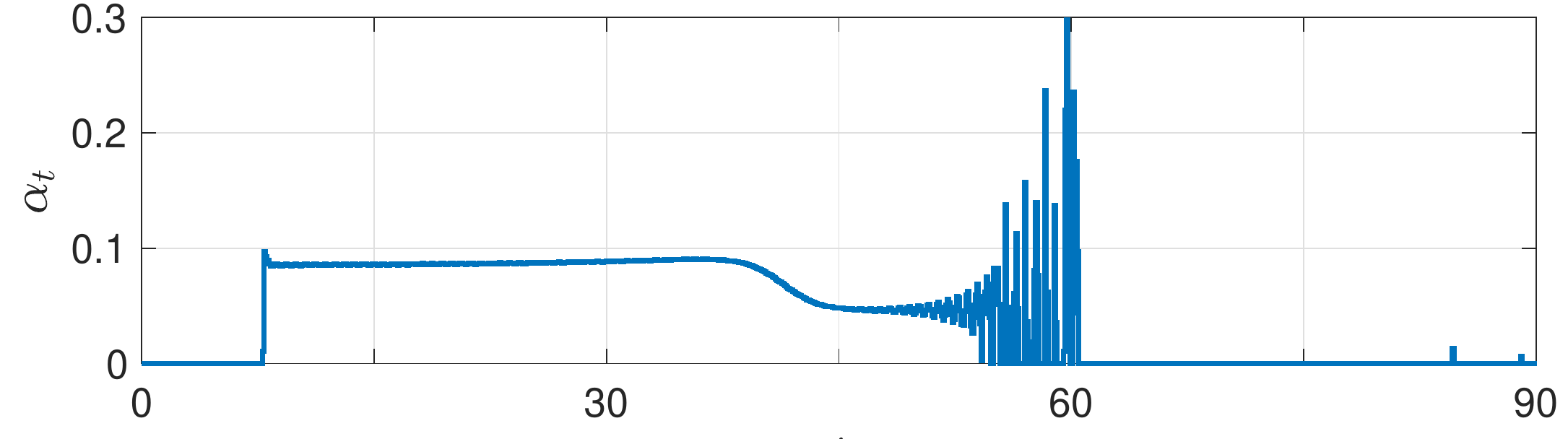}
\label{fig:benchmark1-step}
}
\caption{A comparison of the proposed \ac{esc} controller with an \ac{esc} controller from literature.}
\label{fig:benchmark1}
\end{figure}

Simulation results are shown in Figure~\ref{fig:benchmark1}. Figure~\ref{fig:benchmark1-cost} compares the cost $\cost(y_t)$ profiles of our \ac{esc} controller with the \ac{esc} controller from~\cite{Hunnekens2014}. This figure shows that our \ac{esc} controller converges to the optimal in roughly half the time as the existing controller.

\subsubsection{2-D Benchmark}

In this section, we demonstrate our \ac{esc} controller for the $2$ state benchmark example from~\cite{Suttner2019}. The plant dynamics are
\begin{subequations}
\label{eq:benchmark2-plant}
\begin{align}
\dot{x} &= R(x) u + w \\
y &= x 
\end{align}
\end{subequations}
where $R(x) \in \reals^{2 \times 2}$ is a planar rotation matrix with angle $x_1 + x_2$ and $w(t) = [\sin(2t),\cos(t)]$ is a periodic disturbance. This nonlinear plant~\eqref{eq:benchmark2-plant} does not satisfy Assumption~\ref{assume:plant} since it is not \ac{iss}. Indeed, it is only \textit{marginally stable} for $u=w=0$ and has no equilibrium states $\dot{x} = 0$ for $u \neq 0$ or $w \neq 0$. Thus, we pre-stabilize the system using the controller
\begin{align*}
u = - R(x)^\trans \big( F(x-r) - w \big).
\end{align*}
where the matrix $F = -10 I$ has Hurwitz eigenvalues so that the output will track $y \rightarrow r$ the reference $r = \bar{r}$. To make the problem more challenging and preserve the nonlinearity, we simulate the plant~\eqref{eq:benchmark2-plant}  in continuous-time with the controller updated in discrete-time, i.e., we apply a zero-order hold for the control input $u(t)=u(t_k)$ for $t \in [t_k,t_{k+1})$ which is computed for states $x(t_k)$ and disturbances $w(t_k)$ sampled as discrete-times $t_k$ where $\Delta t = 50$ milliseconds. 

The unknown cost function is
\begin{align}
\label{eq:benchmark2-cost}
\cost(y) = \big\| y - 1 \big \|^2 + 2018.
\end{align}
This strictly convex quadratic cost satisfies our Assumption~\ref{assume:cost}. 

The \ac{esc} controller~\eqref{eq:controller} used gain $K=0.5I$ and a sample-rate of $20$ Hertz. The \ac{bls} estimator~\eqref{eq:estimator} had an estimation horizon of $N=5$ and bounds $\ubar{H} = 0I$ and $\bar{H} = 10I$ on the curvature $\nabla^2 \cost$ of the cost. No dither was used since the periodic disturbance $w(t)$ already provide persistency of excitation. Between sample periods $\Delta t = 0.05$, the nonlinear plant~\eqref{eq:benchmark2-plant} was simulated using MATLAB's \texttt{ode45} solver. 

\begin{figure}[!ht]
\centering
\subfloat[Cost]{
\includegraphics[width = 0.49\columnwidth]{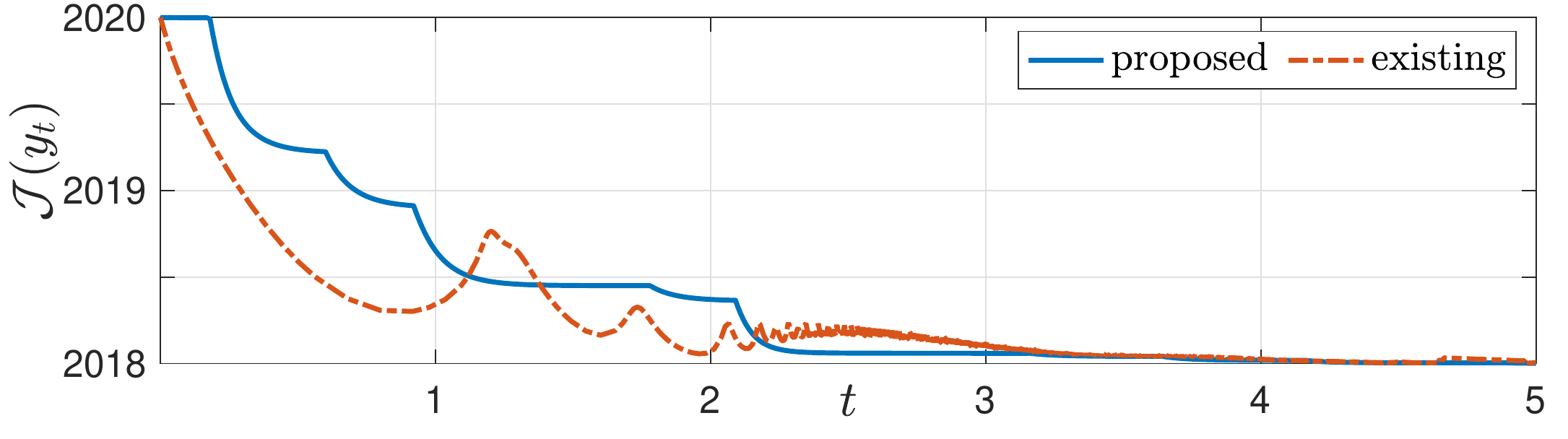}
\label{fig:benchmark2-cost}
}
\subfloat[Step-Size]{
\includegraphics[width = 0.49\columnwidth]{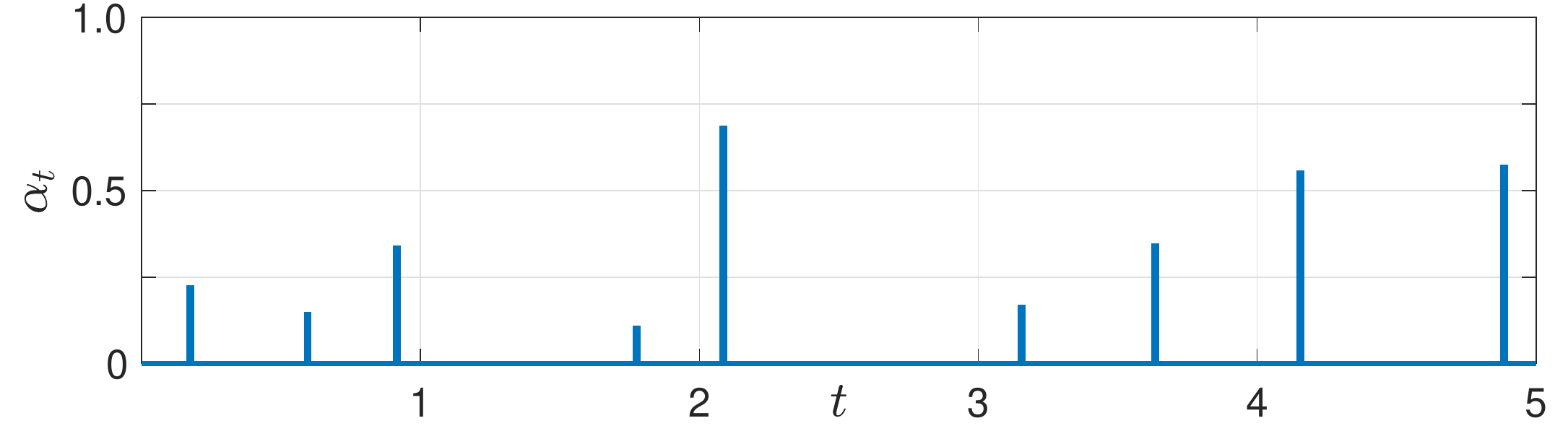}
\label{fig:benchmark2-step}
}
\caption{A comparison of the proposed \ac{esc} controller with an existing \ac{esc} controller from literature.}
\label{fig:benchmark2}
\end{figure}

Simulation results are shown in Figure~\ref{fig:benchmark2}. Figure~\ref{fig:benchmark2-cost} shows that our \ac{esc} controller has comparable performance to the existing controller from~\cite{Suttner2019}. This benchmark example demonstrates how the tracking errors $e_t = r_t - y_t$ affect our adaptive step-size~\eqref{eq:step-size}. Due to the rotation matrix in the dynamics~\eqref{eq:benchmark2-plant}, the plant takes looping paths between the reference set-points $r_t$. This produces highly exciting, but highly non-local $\| r_t - y_t \| \gg 0$ data $\{ \cost_t, y_t\}$, which leads to poor estimates of the gradient $\nabla \cost(r_t)$ at the set-point $r_t$. As a result, the step-size~\eqref{eq:step-size} is almost always zero $\alpha_t = 0$, allowing the plant~\eqref{eq:benchmark2-plant} to settle $y_t \approx r_t$ near the set-point $r_t$ before trusting the estimated gradient. Indeed, the step-size is non-zero $\alpha_t \neq 0$ at only $9$ of the $501$ simulated time instances, as shown in Figure~\ref{fig:benchmark2-step}. Nonetheless, our \ac{esc} controller converged to the optimal with a comparable convergence rate to the specialized \ac{esc} controller from~\cite{Suttner2019}.

\subsubsection{3-D Benchmark}

In this section, we demonstrate our \ac{esc} controller for the $3$ state benchmark example from~\cite{haring2016}. The plant dynamics are
\begin{subequations}
\label{eq:benchmark3-plant}
\begin{align}
\dot{x}_1 &= - x_1 + u_2^2\\
 \dot{x}_2 &= -x_2 + u_1 \\
  \dot{x}_3 &= -x_3 + u_2 x_2.
\end{align}
\end{subequations}
Although the plant~\eqref{eq:plant} does not satisfy our asymptotic tracking assumption, this can be rectified by inverting the steady-state map of the plant using the transformation
\begin{subequations}
\label{eq:benchmark3-gain}
\begin{align}
u_{1} &= r_{1} / (1+\sqrt{r_{2}}) \\ 
u_{2} &= \sqrt{r_{2}}.
\end{align}
\end{subequations}
The plant~\eqref{eq:benchmark3-plant} has an implicit constraint $r_2 \geq 0$, which we enforce by setting $r_2 = 0$ if $r_2 < 0$. This plant~\eqref{eq:benchmark3-plant} is only locally Lipschitz continuous. The unknown cost function is
\begin{align}
\label{eq:benchmark3-cost}
\cost(y) = (x_2 + x_3)^2 + 2(x_1 + x_2 - u_1) = y_1^2 + 2y_2
\end{align}
where $y_1 = x_2 + x_3$ and $y_2 =x_1 + x_2 - u_1$ are the measured outputs. Note that the cost~\eqref{eq:benchmark3-cost} is convex, but not strictly convex. Nonetheless, it satisfies Assumption~\ref{assume:cost}. 

For the \ac{esc} controller~\eqref{eq:controller} design, the plant~\eqref{eq:benchmark3-plant} was converted to discrete-time using the forward Euler method with a sample-time of $\Delta t = 0.25$. The gain $G$ and Lyapunov matrix $P$ were computed using parametric linear matrix inequalities~\cite{Boyd1994} with $u_2 \in [-5,5]$ as the parameter. The \ac{bls} estimator~\eqref{eq:estimator} had an estimation horizon of $N=5$ and bounds $\ubar{H} = 0I$ and $\bar{H} = 10I$ on the curvature $\nabla^2 \cost$ of the cost. The dither $d_t = 0.001[\sin(t), \sin(2t)]^\trans$ was used to provide persistency of excitation. 

\begin{figure}[!ht]
\centering
\subfloat[Cost]{
\includegraphics[width = 0.49\columnwidth]{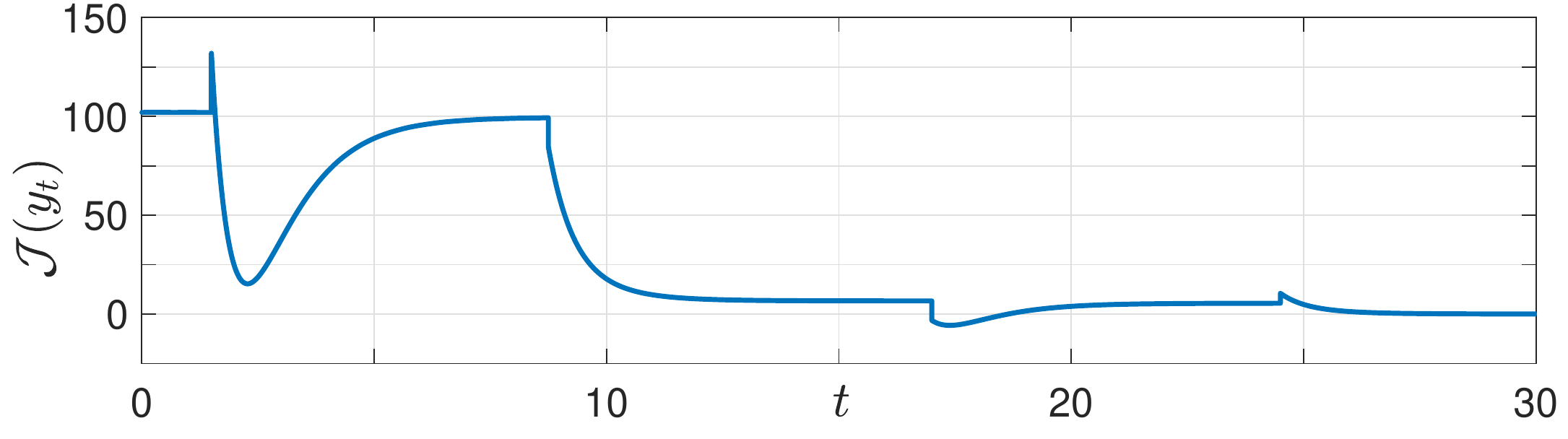}
\label{fig:benchmark3-cost}
}
\subfloat[Step-Size]{
\includegraphics[width = 0.49\columnwidth]{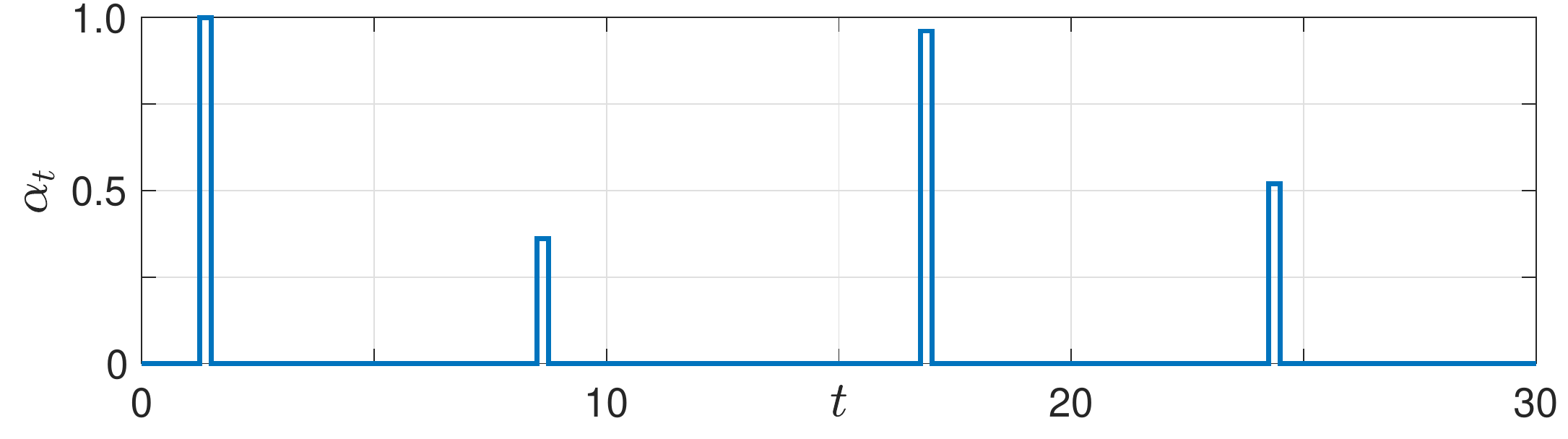}
\label{fig:benchmark3-step}
}
\caption{A comparison of the proposed \ac{esc} controller with an existing \ac{esc} controller from literature.}
\label{fig:benchmark3}
\end{figure}

Simulation results are shown in Figure~\ref{fig:benchmark3}. Between sample periods $\Delta t = 0.25$, the nonlinear plant~\eqref{eq:benchmark3-plant} was simulated using MATLAB's \texttt{ode45} solver. Figure~\ref{fig:benchmark3-cost} shows that our \ac{esc} controller converged to the optimal equilibrium in approximately $30$ seconds, which is approximate $40 \times$ faster than the existing controller. Simulation results for the existing controller are not shown due to the disparity in time-scales. Note that the cost $\cost(y_t)$ can temporarily drop below the optimal equilibrium cost since $y_2 = x_1 + x_2 - u_1 = x_1 - \dot{x}_2$ depends on the state velocity $\dot{x}_2$ which is zero $\dot{x}_2 = 0$ at equilibrium. 

Again, the step-size~\eqref{eq:step-size} is almost always zero $\alpha_t$ as shown in Figure~\ref{fig:benchmark2-step}. The step-size is non-zero $\alpha_t \neq 0$ for only $4$ of the simulated $120$ time instances. In contrast to the previous benchmark example, in this example the mostly zero step-size $\alpha_t =0$ is due to the state velocity $\dot{x}_2$ appearing in the cost. The zero step-size $\alpha_t =0$ allows the plant~\eqref{eq:benchmark3-plant} to settle near an equilibrium $y_t \approx y_{t-1}$ before exploiting the estimated gradient.

\section{Conclusions}

This paper presented an \ac{esc} controller~\eqref{eq:controller} with an adaptive step-size~\eqref{eq:step-size} that adjusts the aggressiveness of the controller based on the quality of the gradient estimate~\eqref{eq:estimator}. We proved that the \ac{bls} estimator~\eqref{eq:estimator} with our novel weighting~\eqref{eq:weighting} produced bounded~\eqref{eq:estimation-error} gradient estimation errors. The adaptive step-size~\eqref{eq:step-size} maximizes the decrease of the Lyapunov function~\eqref{eq:lyapunov} for the worst-case estimation error~\eqref{eq:estimation-error} in the exploitation mode. In the exploration mode, the controller allows the plant to settle improving the gradient estimate. Since the controller~\eqref{eq:controller} interminably re-enters the exploitation mode, we were able to prove that the optimal equilibrium~\eqref{eq:equilibrium} is \ac{iss} for the closed-loop system~\eqref{eq:plant} and~\eqref{eq:controller}.

\bibliographystyle{IEEEtran}
\bibliography{extremum-seeking.bib}

\end{document}